\documentclass{article}[11pt]
\usepackage{graphicx} 
\usepackage{amsfonts,amsmath,amssymb,cite}
\usepackage{amsmath,amsthm}
\usepackage{enumerate}
\usepackage{latexsym}
\usepackage[T1]{fontenc}
\usepackage{color}
\usepackage[dvipsnames]{xcolor}
\usepackage{tikz}
\usetikzlibrary{decorations.pathmorphing}
\usetikzlibrary{shapes.geometric}
\usetikzlibrary{svg.path}
\usetikzlibrary{calc,positioning}
\usepackage[margin=1in]{geometry}
\usepackage{etoolbox}
\usepackage{mathtools}

\usepackage{graphicx}
\usepackage{lineno}
\usepackage{tkz-graph}
\usepackage{tkz-berge}
\usepackage{hyperref}
\usepackage{color}

\hypersetup{colorlinks=true}

\hypersetup{colorlinks=true, linkcolor=blue, citecolor=blue,urlcolor=blue}

\newtheorem{thm}{Theorem}

\newtheorem{ob}{Observation}
\newtheorem{lem}[thm]{Lemma}

\newtheorem{prop}[thm]{Proposition}
\newtheorem{ques}{Question}
\newtheorem{prob}{Problem}

\theoremstyle{definition}

\theoremstyle{remark}
\newtheorem{remark}[thm]{Remark}

\newcommand{\hmk}{\chi_{\mu_{k}}}
\newcommand{\hmd}{\chi_{\mu_{2}}}
\newcommand{\hm}{\chi_{\mu}}

\newcommand{\strong}{\boxtimes}

\newcommand{\mvc}{\chi_{\mu}}

\tikzstyle{vertex}=[circle, draw, inner sep=0pt, minimum size=6pt]

\title{Distance mutual-visibility coloring: relations with (total) domination, exact distance graphs and graph products}

\author{Saneesh Babu$\,^{a,}$\footnote{\tt saneeshbabu1996@gmail.com}, Bo\v{s}tjan Bre\v{s}ar$\,^{b,c,}$\footnote{\tt bostjan.bresar@um.si}, Aparna Lakshmanan S$\,^{a,}$\footnote{\tt aparnals@cusat.ac.in} \, and Babak Samadi$\,^{c,}$\footnote{\tt babak.samadi@imfm.si}\\ \\
$^a$ Department of Mathematics, Cochin University of Science and Technology, India\\
$^b$ Faculty of Natural Sciences and Mathematics, University of Maribor, Slovenia \\
$^c$ Institute of Mathematics, Physics and Mechanics, Ljubljana, Slovenia\\
}
\date{}

\begin{document}
\maketitle

\begin{abstract}
The concept of mutual-visibility introduced by Di Stefano in 2022 has already been extended in several directions. In $k$-distance mutual-visibility, a set $S$ of vertices in a graph $G$ admits the defining condition if for any two vertices in $S$ there exists a geodesic between them of length at most $k$ whose internal vertices are not in $S$. Recently, the mutual-visibility chromatic number of a graph was introduced, and in this paper we combine both ideas as follows. For any positive integer $k$, the $k$-distance mutual-visibility coloring of $G$ is a partition of $V(G)$ into subsets each of which is a $k$-distance mutual-visibility set. The minimum cardinality of such a partition is called the $k$-distance mutual-visibility chromatic number of $G$ and is denoted by $\chi_{\mu_k}(G)$. The special cases $k=1$ and $k\ge {\rm diam}(G)$ lead to the clique cover number $\theta(G)$ and the mutual-visibility chromatic number $\chi_{\mu}(G)$, respectively. So, our attention is given to $1<k<{\rm diam}(G)$ with $k=2$ producing the most interesting results. We prove that $\chi_{\mu_2}(G)\le |V(G)|/2$ and present large families of graphs that attain the bound. In addition, $\chi_{\mu_2}(G)$ is bounded from above by the total domination number $\gamma_t(G)$ if $G$ is isolate-free, while in graphs with girth at least $7$, $\chi_{\mu_2}(G)$ is bounded from below by the domination number $\gamma(G)$. A surprising relation with the exact distance-2 graphs is found, which results in the equality $\theta(G^{[\natural2]})=\gamma_{t}(G)$ holding for any isolate-free graph $G$ with girth at least $7$. 
The relation is explored further in the study of the lexicographic product of graphs, where we prove the sharp inequalities $\chi_{\mu_{2}}(G\circ H)\leq \theta(G^{[\natural2]})\leq \theta\big{(}(G\circ H)^{[\natural2]}\big{)}$. We also prove a sharp lower bound on the $2$-distance mutual-visibility chromatic number in the Cartesian product of two connected graphs, and present the bound $\chi_{\mu_2}(Q_n)\le \gamma(Q_n)$ that holds for all $n$-cubes $Q_n$. Moreover, the canonical upper bound $\chi_{\mu_k}(G \boxtimes H) \leq \hmk(G)\hmk(H)$, where $G \boxtimes H$ is the strong product of any two graphs and $k$ is arbitrary, is also proved and it is widely sharp. Finally, we characterize the block graphs $G$, for which 
$\chi_{\mu_k}(G)=\chi_{\mu}(G)$, where $k = {\rm diam}(G) - 1$. 
\end{abstract} 
\textbf{Math.\ Subj.\ Class. 2020:} 05C12, 05C15, 05C69, 05C76\vspace{0.5mm}\\
{\small \textbf{Keywords:} distance mutual-visibility, distance mutual-visibility coloring, (total) domination, exact distance-2 graph, graph product, geodesic.

\section{Introduction}

In this paper, we consider the coloring of a graph by which its vertices are partitioned into $k$-distance mutual-visibility sets. The concept of $k$-distance mutual-visibility was introduced recently by Cera et al.~\cite{cgvy} as a variation of mutual-visibility introduced by Di Stefano~\cite{DiS} in 2022. The introduction of mutual-visibility in graphs had several motivations. For instance, it is related to the classical no-three-in-line problem on the one hand, and is also close to the problem of repositioning of mobile robots in the plane on the other hand; see~\cite{DiS} for more detail. 
Although mutual-visibility is a recent concept, it has gained a lot of attention; for instance, one can already find over 20 citations of the seminal paper in the MathSciNet.  

The main concepts studied in this paper are defined as follows. Given a graph $G$ and vertices $u,v\in V(G)$, a {\em $u,v$-geodesic} is a shortest path with $u$ and $v$ as its endvertices. The {\em distance}, $d_G(u,v)$, between $u$ and $v$ in $G$ is the length of a $u,v$-geodesic. A subset $M\subseteq V(G)$ is a {\em mutual-visibility set} if for every $u,v\in M$, there exists a $u,v$-geodesic whose internal vertices are not in $M$. The maximum cardinality of a mutual-visibility set in $G$ is the {\em mutual-visibility number} of $G$, denoted $\mu(G)$. The mentioned variation of mutual-visibility from~\cite{cgvy} is defined as follows. Given a graph $G$ and a positive integer $k$, a subset $M\subseteq V(G)$ is a {\em $k$-distance mutual-visibility set} if for every $u,v\in M$, there exists a $u,v$-geodesic of length at most $k$ whose internal vertices are not in $M$. That is, in addition to the condition of mutual-visibility, vertices in $M$ need to be pairwise at distance at most $k$. The maximum cardinality of a $k$-distance mutual-visibility set in $G$ is the {\em $k$-distance mutual-visibility number} of $G$, denoted $\mu_k(G)$. 

In the recent manuscript~\cite{KKVY}, a new version of vertex coloring of a graph $G$ was introduced, which results in a partition of $V(G)$ into mutual-visibility sets. The minimum number of colors for which the graph $G$ admits such a partition is the {\em mutual-visibility chromatic number} of $G$, denoted $\mvc(G)$. Several general properties as well as bounds and exact values in well-known classes of graphs were found for the mutual-visibility chromatic number in~\cite{KKVY}. In addition, the independent version of mutual-visibility coloring was introduced and studied in~\cite{BPSY} continuing the study from~\cite{KKVY}. 

Vertices in a graph with the mutual-visibility property can represent entities in a computer or social network aiming to communicate efficiently and confidentially, ensuring messages do not pass through other entities. A mutual-visibility coloring partitions the vertex set into the minimum number of mutual-visibility sets, where each pair of vertices in a set has a shortest path with no internal vertices in that set. Various such partitions exist. From a practical perspective, if the distance between the vertices in each color partition is less than $k$ for some positive integer $k$, i.e., each partite set is a $k$-distance mutual-visibility set, then this partition is more efficient as the distance between points of communication is less than $k$. So we introduce the following concept.

Given a graph $G$ and a positive integer $k$, the minimum cardinality of a partition of $V(G)$ into $k$-distance mutual-visibility sets is the {\em $k$-distance mutual-visibility chromatic number} of $G$, denoted by $\hmk(G)$. The main purpose of this paper is to study this new coloring invariant, connecting it with the related studies in~\cite{BPSY,KKVY}, finding its general properties and its behavior in well-known classes of graphs. 

In the remainder of this section we provide basic terminology and establish notation, while in the next section we give some preliminary results about the $k$-distance mutual-visibility chromatic number.  In Section~\ref{sec:dom}, we focus on  possibly the most interesting case when $k=2$, and give several bound on the $2$-distance mutual visibility chromatic number. While $\hmd(G)\le \lceil |V(G)|/2\rceil$ holds in any graph $G$, and $\hmd(G)\le \gamma_t(G)$ holds in any graph $G$ with no isolated vertices, we also prove that $\hmd(G)\ge \gamma(G)$ if $G$ has girth greater than $6$. The investigation leads to the surprising result that the clique cover number of the exact distance-2 graph of $G$ is equal to the total domination number of $G$ as soon as $G$ has girth greater than $6$ and no isolated vertices. In Section~\ref{sec:lex}, we continue this line of investigation and for the lexicographic product $G\circ H$ of connected graphs $G$ and $H$ prove that  $\chi_{\mu_{2}}(G\circ H)\leq \theta(G^{[\natural2]})\leq \theta\big{(}(G\circ H)^{[\natural2]}\big{)}$. Under some additional assumptions on $G$ we prove that $\chi_{\mu_{2}}(G\circ H)$ is bounded from above by $\hmd(G)$. In Section~\ref{sec:stong}, we consider the strong product $G\boxtimes H$ of two graphs and prove the sharp upper bound $\hmk(G \boxtimes H) \leq \hmk(G)\hmk(H)$ that holds for any positive integer $k$. 
In Section~\ref{sec:cart}, a sharp lower bound on the $2$-distance mutual-visibility chromatic number is proved, which involves also the $2$-packing numbers of both factors. In addition, we prove that $\hmd(Q_n)\le \gamma(Q_n)$ holds for all $n$-cubes and provide some properties which indicate that the equality might hold. Most of the above mentioned (and other) bounds in this paper are sharp, and we provide different families of graphs that attain them. Finally, in Section~\ref{sec:block}, we characterize the block graphs $G$ for which $\chi_{\mu_{{\rm diam}(G)-1}}(G)=\chi_{\mu}(G)$. 


\subsection{Basic notations and terminologies}

Throughout this paper, we  use the notation $[n]=\{1,\ldots,n\}$ to denote the set of integers from $1$ to $n$. All graphs considered in this paper are simple and undirected. For a graph $G$ and a vertex $u\in V(G)$, the \emph{open neighborhood} of $u$ is denoted by $N_{G}(u)$, and its {\em closed neighborhood} is $N_G[u]=N_G(u)\cup \{u\}$. The {\em diameter} of $G$ is defined by ${\rm diam}(G)=\max\{d_G(u,v):\,u,v\in V(G)\}$. A subgraph $H$ of a graph $G$ is \emph{convex} if for any vertices $x,y\in V(H)$, every $x,y$-geodesic in $G$ lies entirely in $H$. The length of a shortest cycle in a graph $G$ is the {\em girth} of $G$, denoted by $g(G)$. In a forest $F$, we set $g(F)=\infty$. We use~\cite{West} as a reference for terminology and notation which are not explicitly defined here.

Given a graph $G$, a set $D\subseteq V(G)$ is a \textit{dominating set} (resp. \textit{total} {\em dominating set}) in $G$ if
for every $x\in V(G)\setminus D$ \big{(}resp.\ $x\in V(G)$\big{)}, there exists $y\in D\cap N_{G}(x)$. The minimum cardinality of a (total) dominating set in $G$ is the (\textit{total}) {\em domination number} of $G$ and is denoted by $\gamma(G)$ \big{(}resp.\ $\gamma_t(G)$\big{)}. 
A dominating set (resp. total dominating set) $D$ of cardinality $\gamma(G)$ \big{(}resp. $\gamma_{t}(G)$\big{)} is called a $\gamma(G)$-set (resp.\ $\gamma_t(G)$-set).

We abbreviate the expression $k$-distance mutual-visibility set (resp. coloring) to $k$DMV set (resp. coloring), where $k$ can be any positive integer. Given a graph $G$ and a positive integer $k$, by a $\chi_{\mu_{k}}(G)$-coloring and a $\mu_{k}(G)$-set, we mean a $k$DMV coloring and a $k$DMV set with $\chi_{\mu_{k}}(G)$ colors and of cardinality $\mu_{k}(G)$, respectively. 

For graphs $G$ and $H$, their \emph{Cartesian}, \emph{strong} and \emph{lexicographic} product graphs, denoted by $G\square H$, $G\boxtimes H$ and $G\circ H$, respectively, are defined as follows. Each product graph has the vertex set $V(G)\times V(H)$, while the adjacency rules are as follows:
\begin{itemize}
\item in $G \square H$, vertices $(g,h)$ and $(g',h')$ are adjacent if either ``$g=g'$ and $hh'\in E(H)$'' or ``$gg'\in E(G)$ and $h=h'$'';
\item vertices $(g,h)$ and $(g',h')$ in $G \boxtimes H$ are adjacent if either ``one of the two conditions for the Cartesian product holds'' or ``$gg' \in E(G)$ and $hh' \in E(H)$'';
\item in $G \circ H$, vertices $(g,h)$ and $(g',h')$ are adjacent if either $gg'\in E(G)$ or ``$g=g'$ and $hh'\in E(H)$''.
\end{itemize}

In each of these product graphs, the subgraph induced by vertices with the first (resp. second) coordinate fixed is isomorphic to $H$ (resp. $G$). These subgraphs are called $H$-\textit{fibers} and $G$-\textit{fibers}, respectively. More on graph products can be found in~\cite{HIK}.


\section{Preliminary results}

In this section, we list some basic results about the newly introduced concept and related graph invariants.   
We begin by examining the relationship between $\chi_{\mu_k}(G)$ for different $k$, where $1\leq k \leq {\rm diam}(G)$. Clearly, $\chi_{\mu_k}(G)=\chi_\mu(G)$ when $k\ge {\rm diam}(G)$. Moreover, we have $\chi_{\mu_1}(G)=\theta(G)$, where $\theta(G)$ is the {\em clique cover number}, equal to the chromatic number of the complement graph $\overline{G}$ of $G$. Equivalently, $\theta(G)$ is the smallest partition of the vertex set of $G$ into cliques. From the fact that if two vertices are mutually visible in the setting of distance $k$, they remain so when distance $k+1$ is considered, we obtain the following chain of inequalities.

\begin{ob} \label{ob:1}
For any graph $G$ of diameter $d$, 
$$\chi_\mu(G)=\chi_{\mu_d}(G)\leq \chi_{\mu_{d-1}}(G) \leq \ldots\leq \chi_{\mu_1}(G)=\theta(G).$$
\end{ob}
	
Next, we present a trivial lower bound on $\chi_{\mu_k}$ based on the size of a largest $k$DMV set.

\begin{ob} \label{ob:2}
For any graph $G$, $\chi_{\mu_k}(G)\geq \left\lceil \frac{|V(G)|}{\mu_k(G)} \right\rceil$.
\end{ob}

Let $k$ be a positive integer and $G$ be a connected graph. A set $D \subseteq V(G)$ is a \emph{distance $k$-dominating set} of $G$ if every vertex in $V(G)\setminus D$ is within distance $k$ from some vertex in $D$. The \emph{distance $k$-domination number} of $G$, denoted $\gamma_k(G)$, is the minimum cardinality of such a set~\cite{Henning}. We now establish a lower bound for $\hmk(G)$ using this parameter.

\begin{prop}
For any graph $G$, $\gamma_k(G) \leq \hmk(G)$.
\end{prop}
\begin{proof}
Let $\mathcal{S}=\{S_1,\ldots,S_{\hmk(G)}\}$ be a $\hmk(G)$-coloring. Construct a set $D$ by selecting one vertex from each color class $S_i$. Since $\mathcal{S}$ is a partition of $V(G)$ and because $S_{i}$ is a $k$DMV set for each $i\in[\hmk(G)]$, each vertex in $V(G)\setminus D$ is within distance $k$ from a chosen representative in $D$. So, $D$ is a distance $k$-dominating set in $G$. Thus, $\gamma_k(G) \leq|D|=\hmk(G)$.
\end{proof}	

In the following proposition, we determine the exact values of $\chi_{\mu_{k}}$ for all positive integers $k$ in two basic graph classes, namely, paths and cycles.

\begin{prop}\label{str0}
Let $k\geq1$. Then,\vspace{1.5mm}\\
$(i)$ For each $n\geq1$, $\chi_{\mu_{k}}(P_{n})=\left\lceil\frac{n}{2}\right\rceil$.\vspace{1mm}\\
$(ii)$ For each $n\geq3$, $\chi_{\mu_{k}}(C_{n})=\left \{
\begin{array}{lll}
\left\lceil\frac{n}{3}\right\rceil & \mbox{if}\ n\leq3k,\vspace{0.5mm}\\
\left\lceil\frac{n}{2}\right\rceil & \mbox{otherwise}.
\end{array}
\right.$
\end{prop}
\begin{proof}
By Observation \ref{ob:1}, $\chi_{\mu}(G) \leq \chi_{\mu_k}(G) \leq \theta(G)$ for any graph $G$. Since the complete subgraphs of $P_n$ are of cardinality at most 2, it follows that $\theta(P_n)=\lceil\frac{n}{2}\rceil$. On the other hand, $\chi_{\mu}(P_n)=\lceil\frac{n}{2}\rceil$ (see \cite{KKVY}). Hence, $\chi_{\mu_k}(P_n)=\lceil\frac{n}{2}\rceil$.
	
It is proved in \cite{cgvy} that $\mu_k(C_n)=3$ if $n\leq3k$, and $\mu_k(C_n)=2$ otherwise. So, by Observation~\ref{ob:2}, it follows that
\[\chi_{\mu_k}(C_n)\geq \left\lceil \frac{n}{\mu_k(C_n)}\right\rceil=\begin{cases}
\left\lceil \frac{n}{3}\right\rceil & \text{if } n\leq3k,\\
\left\lceil \frac{n}{2}\right\rceil & \text{otherwise}.
\end{cases}\]

To establish equality, it suffices to construct a $k$DMV coloring that achieves these values.
Let $C_{n}:v_{1}v_{2}\ldots v_{n}v_{1}$. For $n>3k$, coloring $v_{2i-1}$ and $v_{2i}$ with $i\in\{1,\ldots,\lfloor\frac{n}{2}\rfloor\}$, and $v_n$ with color $\lceil\frac{n}{2}\rceil$ if $n$ is odd, gives a $k$DMV coloring with $\lceil\frac{n}{2}\rceil$ colors. When $n \leq 3k$, we set $m=\left\lceil \frac{n}{3} \right\rceil$ and define $f: V(C_n)\rightarrow[m]$ by $f(v_i)=f(v_{m+i})=i$ for $i\in[m]$, and $f(v_{2m+i})=i$, for $i\in[n-2m]$. It is readily seen that $f$ is a $k$DMV coloring with $\lceil\frac{n}{3}\rceil$ colors.
\end{proof}

Now, we explore how the $k$DMV chromatic number behaves with respect to convex subgraphs. 
The following result is a variation of analogous results for (independent) mutual-visibility chromatic number given in~\cite[Lemma 1]{BPSY}. 
\begin{prop}\label{prop:convex}
If $H$ is a convex subgraph of a graph $G$, then $\chi_{\mu_k}(G)\geq \chi_{\mu_k}(H)$.
\end{prop}	
\begin{proof}
Fix $k\in \mathbb{N}$. Let $H$ be a convex subgraph of $G$ and let $f:V(G)\to[\ell]$ be a $k$DMV coloring of $G$ with $\hmk(G)$ colors. Due to the convexity of $H$, geodesics between vertices in $H$ are the same regardless of being observed with respect to $H$ or with respect to $G$. Thus, the restriction of $f$ to $H$ is also a $k$DMV coloring of $H$. Therefore, $\hmk(H)\le \ell=\hmk(G)$.     
\end{proof}
    

\section{Relations with the order, (total) domination and exact distance-2 graphs}
\label{sec:dom}

Potentially the most interesting case of $k$DMV coloring is that when $k=2$. In this section, we will prove several results, which confirm this assessment, since $\hmd(G)$ is related to various established graph invariants. 

Given a rooted tree $T$ and any vertex $v\in V(T)$, by $T_{v}$ we mean the subtree of $T$ consisting of $v$ and all its descendants in $T$.

\begin{thm}\label{General}
For any connected graph $G$ of order $n$, $\chi_{\mu_{2}}(G)\leq \big{\lceil}\frac{n}{2}\big{\rceil}$. Moreover, this bound is sharp.
\end{thm}
\begin{proof}
Let $T$ be any spanning tree of $G$. We proceed, by induction on $n$, to show that $\chi_{\mu_{2}}(T)\leq \lceil n/2\rceil$. The result is trivial for $n\in[2]$. On the other hand, the inequality holds when ${\rm diam}(T)=2$ as $\chi_{\mu_{2}}(K_{1,n-1})=2$ for $n\geq3$. So, we may assume that diam$(T)\geq3$.

Assume that the inequality holds for all trees $T'$ of order $n'<n$, and let $T$ be a tree of order $n$. Letting $r$ and $v$ be two leaves of $T$ with $d_{T}(r,v)=$ diam$(T)$, we root the tree $T$ at $r$. Let $u$ be the parent of $v$. Note that $u\neq r$ since diam$(T)\geq3$. 

Assume first that $|V(T_{u})|\geq3$ and that $v$ and $w$ are two leaves adjacent to $u$. We set $T'=T-\{v,w\}$. Using the induction hypothesis, we have $\chi_{\mu_{2}}(T')\leq \lceil(n-2)/2\rceil$. Then, any $\chi_{\mu_{2}}(T')$-coloring can be extended to a $2$DMV coloring $f$ of $T$ by assigning a new color to both $x$ and $y$. Hence, 
\begin{center}
$\chi_{\mu_{2}}(T)\leq|f\big{(}V(T)\big{)}|=\chi_{\mu_{2}}(T')+1\leq \lceil(n-2)/2\rceil+1=\lceil n/2\rceil$.
\end{center}

We now assume that $|V(T_{u})|=2$. Set $T''=T-\{u,v\}$. Then, $\chi_{\mu_{2}}(T'')\leq \lceil(n-2)/2\rceil$ by the induction hypothesis. In this case, any $\chi_{\mu_{2}}(T'')$-coloring can be extended to a $2$DMV coloring of $T$ by assigning a new color to both $u$ and $v$. Therefore,
\begin{center}
$\chi_{\mu_{2}}(T)\leq \chi_{\mu_{2}}(T'')+1\leq \lceil(n-2)/2\rceil+1=\lceil n/2\rceil$.
\end{center}
Thus, $\chi_{\mu_{2}}(T)\leq \big{\lceil}\frac{n}{2}\big{\rceil}$ holds in either case. 

We observe, by definitions, that every $\chi_{\mu_{2}}(T)$-coloring is a $2$DMV coloring of $G$ as well. This results in $\chi_{\mu_{2}}(G)\leq \chi_{\mu_{2}}(T)\leq \big{\lceil}\frac{n}{2}\big{\rceil}$, as desired.

The upper bound is widely sharp. For example, all cycles on at least seven vertices and paths achieve the equality in the upper bound (see Proposition \ref{str0}). Moreover, consider the disjoint union of any graph $H$ with the vertex set $\{v_1,v_2,\ldots,v_s\}$ and $s$ paths $P_{2t_{i}-1}$, for any positive integers $t_{i}$, with $i\in[s]$. Let $G$ be obtained from this union by joining an endvertex of $P_{2t_{i}-1}$ to $v_{i}$ for each $i\in[s]$. Since $d_{G}\big{(}V(P_{2t_{i}-1}),V(P_{2t_{j}-1})\big{)}>2$ for each distinct $i,j\in[s]$, it follows that $\chi_{\mu_{2}}(G)\geq \sum_{i=1}^{s}\chi_{\mu_{2}}(P_{2t_{i}-1})=\sum_{i=1}^{s}t_{i}=|V(G)|/2$. This leads to equality due to the upper bound $\chi_{\mu_{2}}(G)\leq|V(G)|/2$.
\end{proof}

Next, we prove that the total domination number is an upper bound for the $2$DMV chromatic number.  

\begin{thm}
\label{thm:gammatotal}
If $G$ is a graph with no isolated vertices, then $\chi_{\mu_2}(G) \leq \gamma_t(G)$.
\end{thm}
\begin{proof}
Let $D=\{v_1,\ldots,v_{\gamma_{t}(G)}\}$ be a $\gamma_{t}(G)$-set. We construct a partition of $V(G)$ into $\gamma_{t}(G)$ color classes, each forming a $2$DMV set as follows. 
We set $D_{1}=N_{G}(v_{1})$ and $D_{i}=N_{G}(v_{i})\setminus \bigcup_{j=1}^{i-1}N_{G}(v_{j})$ for each $i\in[\gamma_{t}(G)]\setminus \{1\}$. By our choice of $D$, we deduce that $D_{i}\neq \emptyset$ for each $i\in[\gamma_{t}(G)]$. By construction, $D_{1},\ldots,D_{\gamma_{t}(G)}$ are pairwise disjoint. Moreover, they cover $V(G)$ as $D$ is a total dominating set in $G$.
		
We claim that each $D_{i}$ is a $2$DMV set. For any $u,v\in D_{i}$, we have two possibilities:\vspace{-1mm}
\begin{itemize}
\item $uv\in E(G)$; in this case, the path $uv$ is a $u,v$-geodesic of length $1$ with no internal vertices;\vspace{-1mm}
\item $uv\notin E(G)$; in this case, since $u,v\in N(v_{i})$, it follows that $uv_{i}v$ is a $u,v$-geodesic whose internal vertex $v_{i}$ is not in $D_{i}$.	
\end{itemize}\vspace{-1mm}
Thus, each $D_{i}$ is a $2$DMV set in $G$, and so $\{D_{i},\ldots,D_{\gamma_{t}(G)}\}$ is a vertex partition of $G$ into $2$DMV sets. Therefore, $\chi_{\mu_2}(G)\leq \gamma_t(G)$.
\end{proof}

It immediately follows from the following theorem that, if $G$ has girth at least $7$, the domination number is a lower bound for the $2$DMV chromatic number.  
	
\begin{lem}
\label{thm:closedneighborhood}
If $G$ is a graph with $g(G)\geq7$, then every color class $C$ in a $\chi_{\mu_2}(G)$-coloring is a subset of the closed neighborhood of some vertex in $G$. In particular, if $|C|\geq3$, then $C$ is a subset of the open neighborhood of some vertex in $G$.
\end{lem}
\begin{proof}
Let $C$ be a color class in a $\chi_{\mu_2}(G)$-coloring. We show that $C\subseteq N_{G}[v]$ for some vertex $v\in G$. This is trivial when $|C|=1$. Moreover, if $|C|=2$ and $u,v\in C$, then $uv\in E(G)$ or there exists a $u,v$-geodesic $uxv$ such that $x\notin C$. Therefore, $C\subseteq N_{G}[u]$ or $C\subseteq N_{G}[x]$, respectively. So, we assume that $|C|\geq3$ and let $C=\{u_{1},\ldots,u_{|C|}\}$. 

Suppose that $C$ is not independent. Without loss of generality, we may assume that $u_{1}u_{2}\in E(G)$. Since $g(G)\geq7$, $u_3$ is not adjacent to at least one of $u_1$ and $u_2$, say $u_1$. So, there is a $u_{3},u_{1}$-geodesic $u_{3}wu_{1}$ in which $w\notin C$. Again, because $g(G)\geq7$, it follows that $u_{3}u_{2}\notin E(G)$. Therefore, there is a $u_{3},u_{2}$-geodesic $u_{3}xu_{2}$ in $G$ such that $x\notin C$. If $x=w$, then we have the triangle $u_{1}wu_{2}$, which is impossible. Therefore, $x\neq w$. This leads to the existence of the cycle $xu_{3}wu_{1}u_{2}x$ of length $5$, a contradiction. Therefore, $C$ is independent.

Assume that $k\in[|C|]$ is the largest integer such that, by renaming the indices if necessary, $\{u_{1},\ldots,u_{k}\}\subseteq N_{G}(w)$ for some vertex $w\notin C$. Clearly, $k\geq2$. Suppose that $k<|C|$. Since $C$ is a $2$DMV set, there is a $u_{k+1},u_{1}$-geodesic $u_{k+1}w'u_{1}$ such that $w'\notin C$. On the other hand, $w\neq w'$ by our choice of $k$. By the same reason, there exists a $u_{k+1},u_{k}$-geodesic $u_{k+1}w''u_{k}$ such that $w''\notin C$. This leads to the cycle $u_{1}wu_{k}w''u_{k+1}w'u_{1}$ (if $w'\neq w''$) or the cycle $u_{1}wu_{k}w'u_{1}$ (if $w'=w''$), each of them of length less than $7$, a contradiction. Hence, $k=|C|$. Thus, $C\subseteq N_{G}(w)$. This completes the proof.
\end{proof}

As an immediate consequence of Lemma~\ref{thm:closedneighborhood}, we infer the previously mentioned bound. 

\begin{thm}\label{Cor}
If $G$ is a graph of girth at least $7$, then $\chi_{\mu_2}(G)\geq \gamma(G)$.
\end{thm}

It should be noted that $7$ is the smallest possible integer $k$ for which $g(G)\geq k$ implies $\chi_{\mu_2}(G)\geq \gamma(G)$. Fig.~\ref{fig:girth} gives an example of a graph $G$ (in which the numbers indicate the colors given to the vertices in the $\chi_{\mu_2}(G)$-coloring) with $g(G)=6$ and $\gamma(G)=5>\chi_{\mu_2}(G)=3$. 

\begin{figure}[ht!]
\begin{center}
\begin{tikzpicture}[scale=0.7,style=thick,x=1.5cm,y=1.5cm]
\node[circle, draw] (a) at (-0.25,-0.25) {1};
\node[circle, draw] (b) at (2.25,-0.25) {2};
\node[circle, draw] (c) at (3.45,1.8) {1};
\node[circle, draw] (d) at (2.25,3.85) {2}; 
\node[circle, draw] (e) at (-0.25,3.85) {1};  
\node[circle, draw] (f) at (-1.45,1.8) {2};  
\draw (a) -- (b) -- (c) -- (d) -- (e) -- (f) -- (a);

\node[circle, draw] (a1) at (-1.5,-1.8) {3};
\node[circle, draw] (b1) at (3.5,-1.8) {1};
\node[circle, draw] (c1) at (5.4,1.8) {3};
\node[circle, draw] (d1) at (3.5,5.4) {1}; 
\node[circle, draw] (e1) at (-1.5,5.4) {3};  
\node[circle, draw] (f1) at (-3.6,1.8) {1};  
\draw (a1) -- (b1) -- (c1) -- (d1) -- (e1) -- (f1) -- (a1);

\node[circle, draw] (u) at (1,1.8) {1};
\node[circle, draw] (u1) at (0.5,2.7) {2};
\node[circle, draw] (u2) at (2.1,1.8) {2};
\node[circle, draw] (u3) at (0.5,0.9) {2};
\draw (u) -- (u1) -- (e);
\draw (u) -- (u2) -- (c);
\draw (u) -- (u3) -- (a);
\draw (f1) -- (f);
\draw (b1) -- (b);
\draw (d1) -- (d);
\draw (2.13,2.1) .. controls (-0.1,3.8) and (-1,3.75) .. (-3.34,2);
\draw (u3) .. controls (4.25,1.25) and (3.6,5) .. (d1);
\draw (u1) .. controls (0,1.75) and (-0.75,-1) .. (b1);
\end{tikzpicture}
\caption{The graph $G$ with $g(G)=6$, $\gamma(G)=5$ and $\chi_{\mu_2}(G)=3$.}
\label{fig:girth}
\end{center}
\end{figure}
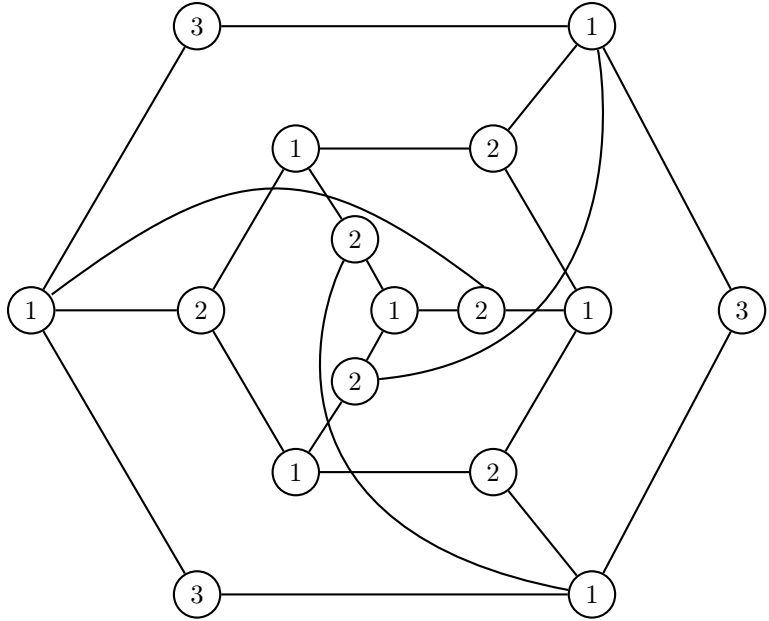




A natural question is whether the difference $\chi_{\mu_2}(G) - \gamma(G)$ can be arbitrarily large in connected graphs with $g(G)\ge 7$. The answer is yes, where the paths serve as examples. For a path $P_n$, $\gamma(P_n) = \lceil \frac{n}{3} \rceil$, whereas $\chi_{\mu_2}(P_n) = \lceil \frac{n}{2} \rceil$.
On the other hand, the difference $\gamma_t(G)-\chi_{\mu_2}(G)$ can also be arbitrarily large in connected graphs with $g(G)\ge 7$. Again it holds even in trees. In particular, for any positive integer $\ell$, there exists a tree $T$ such that $\gamma_{t}(T)-\chi_{\mu_{2}}(T)=\ell$. To see this, consider the double star $S(a,b)$ for integers $a,b\geq2$, with support vertices $x$ and $y$. Let $P_{4}^{i}$ be a copy of the path $P_{4}$ for each $i\in[\ell]$. Now, let $T$ be obtained from $S(a,b)$ by identifying one endvertex of $P_{4}^{i}$ with $x$ for all $i\in[\ell]$. See Fig.~\ref{Fig1}, in which:\\
$\bullet$ $x$ and $y$ are adjacent to $a$ and $b$ leaves, respectively, and\\
$\bullet$ all vertices labeled $i$, for each $i\in[\ell+2]$, form a $2$DMV set in $T$.

It is easy to check that this labeling provides us with a $\chi_{\mu_{2}}(T)$-coloring. On the other hand, it is clear that $\gamma_{t}(T)=2\ell+2$. Therefore, $\gamma_{t}(T)-\chi_{\mu_{2}}(T)=\ell$.

\begin{ob}
For any positive integer $\ell$, there exists a tree $T$ such that $\gamma_{t}(T)-\chi_{\mu_{2}}(T)=\ell$. 
\end{ob}

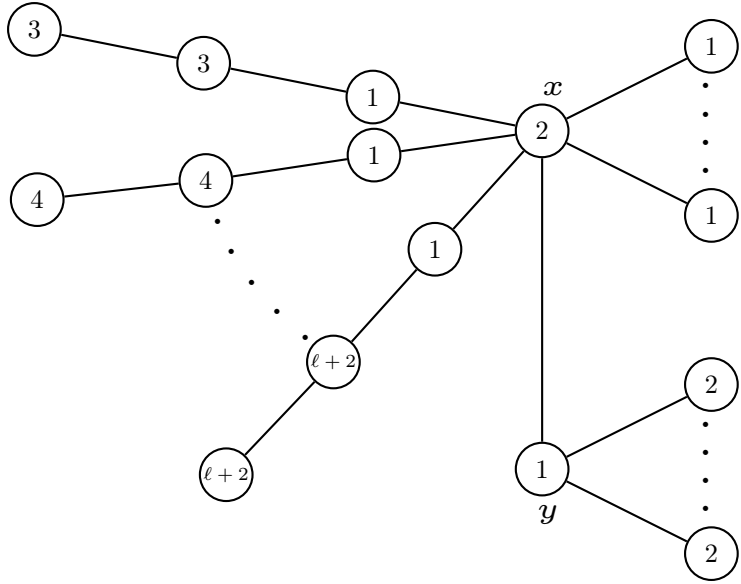
\begin{figure}[ht!]
\centering
\begin{tikzpicture}[scale=0.5, style=thick, x=1.5cm, y=1.5cm ]
  \node[circle, draw, minimum size=0.7cm] (y) at (0,0) {1};
  \node[circle, draw, minimum size=0.7cm] (x) at (0,6) {2};
  \node[circle, draw, minimum size=0.7cm] (y1) at (3,-1.5) {2};
  \node[circle, draw, minimum size=0.7cm] (y2) at (3,1.5) {2};
  \node[circle, draw, minimum size=0.7cm] (x1) at (3,4.5) {1};
  \node[circle, draw, minimum size=0.7cm] (x2) at (3,7.5) {1};
  \node[circle, draw, minimum size=0.7cm] (p11) at (-3,6.6) {1};
  \node[circle, draw, minimum size=0.7cm] (p12) at (-6,7.2) {3};
  \node[circle, draw, minimum size=0.7cm] (p13) at (-9,7.8) {3};
  \node[circle, draw, minimum size=0.7cm] (p21) at (-2.98,5.57) {1};
  \node[circle, draw, minimum size=0.7cm] (p22) at (-5.96,5.12) {4};
  \node[circle, draw, minimum size=0.7cm] (p23) at (-8.95,4.78) {4};
  \node[circle, draw, minimum size=0.7cm] (p31) at (-1.9,3.9) {1};
  \node[circle, draw, minimum size=0.7cm, inner sep=0pt, font=\scriptsize] (p32) at (-3.7,1.9) {$\ell+2$};
  \node[circle, draw, minimum size=0.7cm, inner sep=0pt, font=\scriptsize] (p33) at (-5.6,-0.1) {$\ell+2$};

  \draw (x) -- (y);
    \draw (y1) -- (y) -- (y2);
      \draw (x1) -- (x) -- (x2);
        \draw (x) -- (p11) -- (p12) -- (p13);
           \draw (x) -- (p21) -- (p22) -- (p23);
            \draw (x) -- (p31) -- (p32) -- (p33);
   \node[scale=3.25]  at (-5.75,4.4) {\tiny $.$};   
  \node[scale=3.25]  at (-5.51,3.87) {\tiny $.$};   
\node[scale=3.25]  at (-5.15,3.32) {\tiny $.$};  
\node[scale=3.25]  at (-4.7,2.79) {\tiny $.$};   
\node[scale=3.25]  at (-4.2,2.34) {\tiny $.$};   
\node[scale=3.25]  at (2.9,6.8) {\tiny $.$};   
\node[scale=3.25]  at (2.9,6.3) {\tiny $.$};   
\node[scale=3.25]  at (2.9,5.8) {\tiny $.$};   
\node[scale=3.25]  at (2.9,5.3) {\tiny $.$};   
\node[scale=3.25]  at (2.9,0.8) {\tiny $.$};   
\node[scale=3.25]  at (2.9,0.3) {\tiny $.$};   
\node[scale=3.25]  at (2.9,-0.2) {\tiny $.$};   
\node[scale=3.25]  at (2.9,-0.7) {\tiny $.$};   

  \node[scale=2.25] at (0.2,6.75) {\tiny $x$};
   \node[scale=2.25] at (0.1,-0.80) {\tiny $y$};
  
\end{tikzpicture}

\caption{{\small A tree $T$ with $\gamma_{t}(T)=2\ell+2$ and $\chi_{\mu_{2}}(T)=\ell+2$. Here, $x$ (resp. $y$) is adjacent to $a$ (resp. $b$) leaves and it is an endvertex of $\ell$ copies of the path $P_{4}$.}}\label{Fig1}
\end{figure}



At the end of this section, we turn our attention to relations with exact distance graphs. 
Given a graph $G$, the \textit{exact distance-$p$ graph} $G^{[\natural p]}$ has $V(G)$ as its vertex set, and two vertices are adjacent whenever the distance between them in $G$ equals $p$. Note that the exact distance graphs have been extensively studied in the literature with the main focus on their chromatic number. They were introduced and studied already in the 1980's~\cite{Sim83}, and were then rediscovered a decade ago~\cite[Section 11.9]{NESE}, which initiated several further studies; see some recent papers~\cite{BOUS, BGKT, HEUV,Q}. 

In the following result, the $1$DMV number appears, in the form of the clique cover number. In the proof, we will use the concept of independent $2$-distance mutual-visibility defined as follows. A set $D$ is an {\em independent $2$-distance mutual-visibility set}, or an I$2$DMV {\em set} for short, if $D$ is independent and $2$DMV set at the same time. The cardinality of a minimum I$2$DMV set in $G$ is denoted by $\chi_{i\mu_2}(G)$. 

\begin{thm}\label{thm:dis}
If $G$ is a graph with $g(G)\geq7$ and no isolated vertices, then $\theta(G^{[\natural2]})= \chi_{i\mu_2}(G) = \gamma_{t}(G)$.
\end{thm}
\begin{proof}
Let $D,D_{1},\ldots,D_{\gamma_{t}(G)}$ be the same as those in the proof of Theorem \ref{thm:gammatotal}. We claim that $D_{i}$ is an I$2$DMV set in $G$ for every $i\in[\gamma_{t}(G)]$. As it was proved in the proof of Theorem \ref{thm:gammatotal}, $D_{i}$ is a $2$DMV set in $G$ for each $i\in[\gamma_{t}(G)]$. On the other hand, for any distinct vertices $u,v\in D_{i}$, we observe that $uv\notin E(G)$ because $uv_{i},vv_{i}\in E(G)$ and $g(G)\geq7$. Thus, each $D_{i}$ is an I$2$DMV set in $G$, and so $\{D_{i},\ldots,D_{\gamma_{t}(G)}\}$ is a vertex partition of $G$ into I$2$DMV sets. Therefore,
\begin{equation}\label{Total} 
\chi_{i\mu_2}(G)\leq \gamma_{t}(G).
\end{equation}

Note that every color class in a $\chi_{i\mu_2}(G)$-coloring is a subset of the open neighborhood of some vertex in $G$. To see this, one can use the same arguments as in the proof of Lemma~\ref{thm:closedneighborhood} noting that an I$2$DMV set is also an $2$DMV set while at the same time two vertices in the same color class cannot be adjacent. Thus, letting $\mathcal{Q}$ be a $\chi_{i\mu_2}(G)$-coloring, every set $Q\in \mathcal{Q}$ is a subset of $N_{G}(v)$ for some $v\in V(G)$. For every $Q\in \mathcal{Q}$, we choose only one such vertex $v$, and let $S$ be the set of such vertices $v$. Then, it is clear that $|S|\leq|\mathcal{Q}|$. Moreover, $S$ is a total dominating set in $G$ by its definition and since $\mathcal{Q}$ is   a vertex partition of $G$. Therefore, $\gamma_{t}(G)\leq|S|\leq|\mathcal{Q}|=\chi_{i\mu_2}(G)$. This leads to the equality 
\begin{equation}\label{Total2} 
\chi_{i\mu_2}(G)=\gamma_t(G)
\end{equation}
due to the inequality (\ref{Total}).

We now turn our attention to the exact distance-$2$ graph $G^{[\natural2]}$. It is readily checked, by definitions, that a vertex subset $A$ in $G$ is an I$2$DMV set if and only if it is a clique in the graph $G^{[\natural2]}$. This in particular implies that $\chi_{i\mu_2}(G)=\theta(G^{[\natural2]})$. Consequently, we have the desired equality $\theta(G^{[\natural2]})=\gamma_{t}(G)$ in view of (\ref{Total2}). 
\end{proof}

Note that Theorem~\ref{thm:dis} is best possible in the sense that it does not hold for graphs $G$ with $g(G)=6$ with no isolated vertices. To see this, let $G$ be obtained from the cycle $C_{6}:v_{1}v_{2}v_{3}v_{4}v_{5}v_{6}v_{1}$ by joining new vertices $u_{1},u_{3},u_{5}$ to $v_{1},v_{3},v_{5}$, respectively. It is then easy to see that $\gamma_{t}(G)=5$ and $\theta(G^{[\natural2]})=4$.

It is worth mentioning, due to the proof of Theorem~\ref{thm:dis}, that the inequality $\theta(G^{[\natural2]})\leq \gamma_{t}(G)$ holds for all triangle-free graphs $G$ with no isolated vertices. However, the triangle-free property cannot be omitted here. To see this, it suffices to consider the complete graph $K_{n}$, for $n\geq3$, with $\gamma_{t}(K_{n})=2$ and $\theta(K_{n}^{[\natural2]})=n$.

\section{Lexicographic product of graphs}
\label{sec:lex}

In our first upper bound on the $2$DMV chromatic number of the lexicographic product of graphs, the concept of exact 2-distance graph is again instrumental. 

\begin{thm}\label{Lexic}
If $G$ is a connected graph on at least two vertices and $H$ is any graph, then 
\begin{center}
 $\chi_{\mu_{2}}(G\circ H)\leq \theta(G^{[\natural2]})\leq \theta\big{(}(G\circ H)^{[\natural2]}\big{)}$.
\end{center}
These inequalities are sharp. 
\end{thm}
\begin{proof}
Let $\mathcal{Q}=\{Q_{1},\ldots,Q_{|\mathcal{Q}|}\}$ be a $\chi_{i\mu_{2}}(G)$-coloring. We set $\mathcal{Q}'=\{Q_{1}\times V(H),\ldots,Q_{|\mathcal{Q}|}\times V(H)\}$. Clearly, $\mathcal{Q}'$ is a vertex partition of $G\circ H$. Let $Q_{i}\times V(H)$ be any set in $\mathcal{Q}'$ and $(g,h),(g',h')\in Q_{i}\times V(H)$ be distinct vertices. 

Assume first that $g\neq g'$. We observe that $(g,h)(g',h')\notin E(G\circ H)$ as $Q_{i}$ is independent. In such a situation, since $Q_{i}$ is an I$2$DMV set in $G$, it follows that there exists a $g,g'$-geodesic $gxg'$ in $G$ such that $x\notin Q_{i}$. This results in the existence of the $(g,h),(g',h')$-geodesic $(g,h)(x,h)(g',h')$ in $G\circ H$ whose internal vertex does not belong to $Q_{i}\times V(H)$. Now, let $g=g'$. If $hh' \in E(H)$, then $(g,h)(g',h') \in E(G \circ H)$. Therefore, consider $hh'\notin E(H)$. Since $G$ has no isolated vertices and because $Q_{i}$ is an independent set in $G$, there is a vertex $g''\notin Q_{i}$ such that $gg''\in E(G)$. This leads, by the adjacency rule of $G\circ H$, to the $(g,h)(g',h')$-geodesic $(g,h)(g'',h)(g',h')$ in $G\circ H$ such that $(g'',h)\notin Q_{i}\times V(H)$. Summing up, we have proved that $Q_{i}\times V(H)$ is a $2$DMV set in $G\circ H$ for each $i\in[|\mathcal{Q}|]$. Therefore, $\mathcal{Q}'$ is a partition of $V(G\circ H)$ into $2$DMV sets, and hence 
\begin{center}
$\chi_{\mu_{2}}(G\circ H)\leq|\mathcal{Q}'|=|\mathcal{Q}|=\chi_{i\mu_{2}}(G)=\theta(G^{[\natural2]})$.
\end{center}

To prove the second inequality, we let $\mathcal{A}=\{A_{1},\ldots,A_{|\mathcal{A}|}\}$ be $\chi_{i\mu_{2}}(G\circ H)$-coloring. Consider the sets $\pi_{G}(A_{1}),\ldots,\pi_{G}(A_{|\mathcal{A}|})$, in which $\pi_{G}$ is the projection map onto $G$ \big{(}that is, $\pi_{G}:V(G\circ H)\rightarrow V(G)$ defined by $\pi_{G}\big{(}(g,h)\big{)}=g$ for each $(g,h)\in V(G\circ H)$\big{)}. It is clear that $V(G)=\bigcup_{i=1}^{|\mathcal{A}|}\pi_{G}(A_{i})$ as $V(G\circ H)=\bigcup_{i=1}^{|\mathcal{A}|}A_{i}$.
We now set $B_{1}=\pi_{G}(A_{1})$ and $B_{i}=\pi_{G}(A_{i})\setminus \bigcup_{j=1}^{i-1}\pi_{G}(A_{j})$ for each $i\in[|\mathcal{A}|]\setminus \{1\}$. It is then clear that the sets $B_{i}$ are pairwise disjoint. Moreover, $V(G)=\bigcup_{i=1}^{|\mathcal{A}|}B_{i}$.  

For any $i\in[|\mathcal{A}|]$, let $g,g'\in B_{i}$ be distinct vertices. Then, there exist $h,h'\in V(H)$ such that $(g,h),(g',h')\in A_{i}$. If $gg'\in E(G)$, then $(g,h)(g',h')\in E(G\circ H)$ by the adjacency rule of $G\circ H$. This is impossible as $A_{i}$ is an I$2$DMV set in $G\circ H$. Therefore, $B_{i}$ is an independent set in $G$. 

On the other hand, since $A_{i}$ is an I$2$DMV set in $G\circ H$, it follows that there is a $(g,h),(g',h')$-geodesic $(g,h)(x,y)(g',h')$ in $G\circ H$ such that $(x,y)\notin A_{i}$. Moreover, we necessarily have $g\neq x\neq g'$. 
Therefore, $gxg'$ is a path in $G$. Suppose to the contrary that $x\in B_{i}$. In such a situation, there exists a vertex $z\notin V(H)\setminus \{y\}$ such that $(x,z)\in A_{i}$. Since $gx\in E(G)$, this in particular implies that $(g,h)(x,z)\in E(G\circ H)$. This is a contradiction as $A_{i}$ is an independent set in $G\circ H$. Therefore, $x\notin B_{i}$. We now deduce from the above argument that $gxg'$ is a $g,g'$-geodesic in $G$ whose internal vertex does not belong to $B_{i}$. In fact, we have proved that $B_{i}$ is an I$2$DMV set in $G$. This shows that $\mathcal{B}=\{B_{1},\ldots,B_{|\mathcal{A}|}\}\setminus \{\emptyset\}$ is a vertex partition of $G$ into I$2$DMV sets. Thus, 
\begin{center}
$\theta(G^{[\natural2]})=\chi_{i\mu_{2}}(G)\leq|\mathcal{B}|\leq|\mathcal{A}|=\chi_{i\mu_{2}}(G\circ H)=\theta\big{(}(G\circ H)^{[\natural2]}\big{)}$.
\end{center}

To see that the first inequality is sharp, let $G\cong K_{1,r}$ for some integer $r\geq2$, and $H\cong P_{t}$ for some integer $t\geq2$. It is then clear that $\theta(K_{1,r}^{[\natural2]})=\chi_{i\mu_{2}}(K_{1,r})=2$. Thus, by the proved inequality $\chi_{\mu_{2}}(K_{1,r}\circ P_{t})\leq \theta(K_{1,r}^{[\natural2]})$, we get $\chi_{\mu_{2}}(K_{1,r}\circ P_{t})\le 2$. On the other hand, since $K_{1,r}\circ P_{t}$ is not a complete graph, we infer $\chi_{\mu_{2}}(K_{1,r}\circ P_{t})=2$, as desired.

That the second inequality is sharp, may be seen as follows. Let $G\cong K_{1,r}$ and $H\cong \overline{K_{t}}$ for any positive integers $r$ and $t\ge 2$. (Note that the graph $K_{1,r}\circ \overline{K_{t}}$ is still connected.) It is readily checked that $\{v\}\times V(\overline{K_{t}})$ and $V(K_{1,r}\circ \overline{K_{t}})\setminus(\{v\}\times V\big{(}\overline{K_{t}})\big{)}$ form a vertex partition of $K_{1,r}\circ \overline{K_{t}}$ into I$2$DMV sets, where $v$ is the center of $K_{1,r}$. Hence, $\theta\big{(}(K_{1,r}\circ \overline{K_{t}})^{[\natural2]}\big{)}=\chi_{i\mu_{2}}(K_{1,r}\circ \overline{K_{t}})\leq2$. On the other hand, $\chi_{i\mu_{2}}(K_{1,r}\circ \overline{K_{t}})\geq2$ since the graph $K_{1,r}\circ \overline{K_{t}}$ has more than one vertex. This completes the proof.
\end{proof}

The first inequality in Theorem \ref{Lexic} can be improved if we assume that $\delta(G)\geq2$ and $g(G)\geq5$. 

\begin{thm}\label{con}
If $G$ is a graph with $\delta(G)\geq2$ and $g(G)\geq5$, and $H$ is any graph, then $$\chi_{\mu_{2}}(G\circ H)\leq \chi_{\mu_{2}}(G).$$ This bound is sharp.
\end{thm}
\begin{proof}
Let $\mathcal{A}=\{A_{1},\ldots,A_{|\mathcal{A}|}\}$ be any $\chi_{\mu_{2}}(G)$-coloring. Consider $\mathcal{B}=\{A_{1}\times V(H),\ldots,A_{|\mathcal{A}|}\times V(H)\}$ and let $(g,h),(g',h')$ be distinct vertices in $A_{i}\times V(H)$ for any $i\in[|\mathcal{A}|]$. Without loss of generality, we may assume that $(g,h)(g',h')\notin E(G\circ H)$. Assume first that $g\neq g'$. This implies, by the adjacency rule of $G\circ H$, that $gg'\notin E(G)$. Due to this, since $A_{i}$ is a $2$DMV set in $G$, there exists a $g,g'$-geodesic $gg''g$ in $G$ such that $g''\notin A_{i}$. This in particular leads to the $(g,h),(g',h')$-geodesic $(g,h)(g'',h)(g',h')$, where $(g'',h)\notin A_{i}\times V(H)$.

Now, we assume that $g=g'$. Suppose that $N_{G}[g]\subseteq A_{i}$. Using this together with $\delta(G)\geq2$, we have $|A_{i}|\geq3$, and let $\{x,y\}\subseteq A_i\cap N_G(g)$. Because $g(G)\geq5$, it follows that $xy\notin E(G)$. Since $A_{i}$ is a $2$DMV set in $G$, there is an $x,y$-geodesic $xzy$ in $G$ whose internal vertex $z$ does not belong to $A_{i}$. This leads to the cycle $xzygx$ in $G$, which is a contradiction to the fact that $g(G)\geq5$. Thus, $g$ has a neighbor $x\in V(G)\setminus A_{i}$. Hence $(g,h)(x,h)(g',h')$ is a $(g,h),(g',h')$-geodesic in $G\circ H$ such that $(x,h)\notin A_{i}\times V(H)$.

In either case, we have proved that any two nonadjacent vertices in $A_{i}\times V(H)$ are connected by a geodesic, of length at most two, in $G\circ H$ whose internal vertex does not belong to $A_{i}\times V(H)$. Therefore, $\mathcal{B}$ is a partition of $V(G\circ H)$ into $2$DMV sets. Thus, $\chi_{\mu_{2}}(G\circ H)\leq|\mathcal{B}|=|\mathcal{A}|=\chi_{\mu_{2}}(G)$.

To see that the bound is sharp, let $G$ be any graph with $\delta(G)\geq2$, $g(G)\geq5$ and $\chi_{\mu_{2}}(G)=2$. (Note that such graphs exist. For example, consider $G\in \{C_{5},C_{6}\}$ or let $G$ be obtained from the cycle $v_{1}v_{2}\ldots,v_{8}v_{1}$ by adding the chords $v_{1}v_{5}$ and $v_{3}v_{7}$.) Then, for any graph $H$, we have $\chi_{\mu_{2}}(G\circ H)=2$ due to the bound and the fact that $G\circ H$ is not a complete graph.
\end{proof}

It should be noted that the conditions in Theorem \ref{con} cannot be relaxed. In fact, we can even prove a stronger statement as follows.

\begin{prop}\label{pr}
Theorem \ref{con} is best possible in the sense that $s=2$ and $t=5$ are the smallest possible integers such that $\delta(G)\geq s$ and $g(G)\geq t$ results in the inequality in Theorem~\ref{con}.
\end{prop}
\begin{proof}
Consider the graph $G$ depicted in Fig.~\ref{convexP3}. Clearly, $g(G)=4$. We observe that the labeling in the figure provides an optimal $2$DMV coloring with two colors. Now, consider the lexicographic product graph $G\circ P_{3}$. Suppose to the contrary that $\chi_{\mu_2}(G\circ P_{3})=2$ and that $f$ is a $\chi_{\mu_2}(G\circ P_{3})$-coloring.
Let $P_{3}:v_{1}v_{2}v_{3}$ and $\alpha,\beta,\gamma\in \{1,2\}$. For the sake of convenience, for any vertex $g\in V(G)$, by assigning $\alpha\beta\gamma$ to $g$ we mean the colors $\alpha,\beta,\gamma$ are assigned to the vertices $(g,v_{1}),(g,v_{2}),(g,v_{3})\in V(G\circ P_{3})$, respectively, under $f$.

Since $d_{G}(b,e)=d_{G}(c,e)=d_{G}(d,e)>2$, the adjacency rule of $G\circ P_{3}$ implies that
\begin{center}
$f\Big{(}\{e\}\times V(P_{3})\Big{)}\cap\Big{(}f\big{(}\{b\}\times V(P_{3})\big{)}\cup f\big{(}\{c\}\times V(P_{3})\big{)}\cup f\big{(}\{d\}\times V(P_{3})\big{)}\Big{)}=\emptyset$.
\end{center}
Due to this, and because $\chi_{\mu_2}(G\circ P_{3})=2$, we may assume without loss of generality that $111$ is assigned to $b,c,d$ and $222$ is assigned to $e$ under $f$. With this in mind, since $d_{G}(b,d_{11})=d_{G}(b,d_{12})=d_{G}(d,b_{11})=d_{G}(d,b_{12})>2$, it necessarily follows that $f(b_{11})=f(b_{12})=f(d_{11})=f(d_{12})=222$.

\begin{figure}[ht!]
\begin{center}
\begin{tikzpicture}[scale=0.5,style=thick,x=2cm,y=2cm]
\node[circle, draw] (a) at (-0.25,0) {2};
\node[circle, draw] (a1) at (-3.25,2) {1};
\node[circle, draw] (a2) at (-0.25,2) {1};
\node[circle, draw] (a3) at (2.75,2) {1}; 
\node[circle, draw] (b1) at (-4,3.5) {2};  
\node[circle, draw] (b2) at (-2.5,3.5) {2};  
\node[circle, draw] (b3) at (-1,3.5) {2}; 
\node[circle, draw] (b4) at (0.5,3.5) {2}; 
\node[circle, draw] (b5) at (2,3.5) {2}; 
\node[circle, draw] (b6) at (3.5,3.5) {2}; 
\node[circle, draw] (b135) at (-3,6) {1}; 
\node[circle, draw] (b146) at (-1.25,6) {1};
\node[circle, draw] (b236) at (0.75,6) {1};
\node[circle, draw] (b245) at (2.5,6) {1};
\node[circle, draw] (c) at (-0.25,7.5) {2};
\draw (a)--(a1)--(a2)--(a3)--(a);
\draw (b1)--(a1)--(b2);
\draw (b3)--(a2)--(b4);
\draw (b5)--(a3)--(b6);
\draw (b1)--(a1)--(b2);
\draw (b1)--(b135)--(b2);
\draw (b3)--(b135)--(b4);
\draw (b5)--(b135)--(b6);
\draw (b1)--(b146)--(b2);
\draw (b3)--(b146)--(b4);
\draw (b5)--(b146)--(b6);
\draw (b1)--(b236)--(b2);
\draw (b3)--(b236)--(b4);
\draw (b5)--(b236)--(b6);
\draw (b1)--(b245)--(b2);
\draw (b3)--(b245)--(b4);
\draw (b5)--(b245)--(b6);
\draw (a) .. controls (-6,0.5) and (-5,5) .. (b135);
\draw (a) .. controls (5.5,0.5) and (4.5,5) .. (b245);

\node (a) at (-0.25,-0.45) {$a$};
\node (a) at (-3.25,1.52) {$b$};
\node (a) at (-0.25,1.52) {$c$};
\node (a) at (2.75,1.52) {$d$};
\node (a) at (-3.95,3) {$b_{11}$};
\node (a) at (-2.45,3) {$b_{12}$};
\node (a) at (-0.99,3) {$c_{11}$};
\node (a) at (0.55,3) {$c_{12}$};
\node (a) at (2.02,3) {$d_{11}$};
\node (a) at (3.55,3) {$d_{12}$};
\node (a) at (-3.05,6.47) {$e_{1}$};
\node (a) at (-1.3,6.47) {$e_{2}$};
\node (a) at (0.8,6.47) {$e_{3}$};
\node (a) at (2.55,6.47) {$e_{4}$};
\node (a) at (-0.25,7.95) {$e$};

\draw (b135)--(c)--(b146);
\draw (b236)--(c)--(b245);
\end{tikzpicture}
\caption{A $\chi_{\mu_2}(G)$-coloring of the graph $G$ with $2$ colors.}
\label{convexP3}
\end{center}
\end{figure}
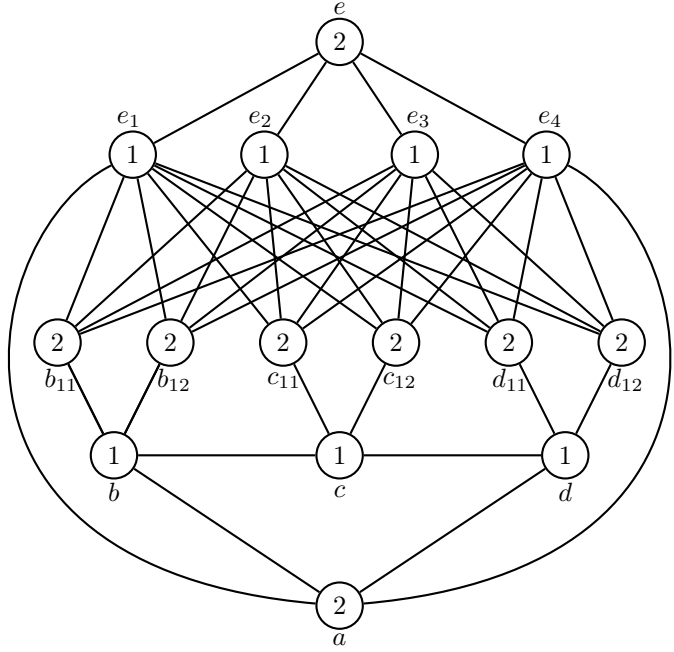

Since $f$ is a $2$DMV coloring of $G\circ P_{3}$ and because $(c,v_{1})$ and $(c,v_{3})$ are two nonadjacent vertices with color $1$, it follows that they must have a common neighbor with color $2$. Without loss of generality, we may assume that $(c_{12},v_{i})$ is such a neighbor for some $i\in[3]$. Now, because $d_{G}(a,c_{12})>2$, it necessarily follows that $f(a)=111$ and $f(c_{12})=f(c_{11})=222$. On the other hand, $f(a)=111$ and $d_{G}(a,e_{2})=d_{G}(a,e_{3})>2$ result in $f(e_{2})=f(e_{3})=222$. This is a contradiction because $(e_{2},v_{1})$ and $(e_{2},v_{3})$, with $f\big{(}(e_{2},v_{1})\big{)}=f\big{(}(e_{2},v_{3})\big{)}=2$, are two nonadjacent vertices in $G\circ P_{3}$ with no common neighbor colored $1$ under $f$. Hence, $\chi_{\mu_2}(G\circ P_{3})\geq3$. (One can prove that, in fact, $\chi_{\mu_2}(G\circ P_{3})=3$). In any case, we derive that $$\chi_{\mu_2}(G\circ P_{3})>\chi_{\mu_2}(G).$$
 
For a graph $G$ with girth $3$, take $G\cong K_{3}$, and let $H$ be any non-complete graph on at least $2$ vertices. Clearly, $\chi_{\mu_2}(K_{3})=1$, while $\chi_{\mu_2}(K_{3}\circ H)\geq2$ as $K_{3}\circ H$ is not a complete graph. 

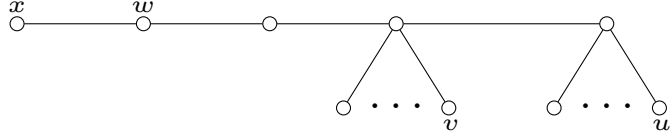
\begin{figure}[ht!]
\centering
\begin{tikzpicture}[scale=0.56, transform shape]

\node [draw, shape=circle] (x) at (0,0) {};
\node [draw, shape=circle] (y) at (5,0) {};
\draw (x)--(y);

\node [draw, shape=circle] (x1) at (-1.25,-2) {};
\node [draw, shape=circle] (x2) at (1.25,-2) {};
\draw (x1)--(x)--(x2);

\node [draw, shape=circle] (y1) at (3.75,-2) {};
\node [draw, shape=circle] (y2) at (6.25,-2) {};
\draw (y1)--(y)--(y2);

\node [draw, shape=circle] (u) at (-3,0) {};
\node [draw, shape=circle] (v) at (-6,0) {};
\node [draw, shape=circle] (w) at (-9,0) {};

\draw (x)--(u)--(v)--(w);

\node [scale=3.25] at (4.5,-1.95) {\small $.$};
\node [scale=3.25] at (5,-1.95) {\small $.$};
\node [scale=3.25] at (5.5,-1.95) {\small $.$};
\node [scale=3.25] at (-0.5,-1.95) {\small $.$};
\node [scale=3.25] at (0,-1.95) {\small $.$};
\node [scale=3.25] at (0.5,-1.95) {\small $.$};

\node [scale=3.25] at (6.3,-2.4) {\tiny $u$};
\node [scale=3.25] at (1.3,-2.4) {\tiny $v$};
\node [scale=3.25] at (-6,0.4) {\tiny $w$};
\node [scale=3.25] at (-9,0.4) {\tiny $x$};

\end{tikzpicture}
\caption{{\small The tree $T$, with and $\chi_{\mu_{2}}(T)=3$, mentioned in the proof of Proposition \ref{pr}.}}\label{Fig4}
\end{figure}

Finally, let $T$ be the tree depicted in Fig.~\ref{Fig4}, in which the vertices $u$, $v$, $w$ and $x$ are distinguished. (Note that  $g(T)=\infty$, while $\delta(T)=1$.) It is easy to check that $\chi_{\mu_2}(T)=3$. Now, we consider the graph $T\circ \overline{K_{2}}$, in which $V(\overline{K_{2}})=\{a,b\}$. Suppose that $\chi_{\mu_2}(T\circ \overline{K_{2}})=3$ and let $f$ be a $\chi_{\mu_2}(T\circ \overline{K_{2}})$-coloring. Since any two vertices from $\{(u,a),(v,a),(w,a)\}$ are at distance at least $3$ apart, we may assume that $f\big{(}(u,a)\big{)}=1$, $f\big{(}(v,a)\big{)}=2$ and $f\big{(}(w,a)\big{)}=3$. In a similar fashion, and since $f$ assigns precisely $3$ colors to the vertices of $T\circ \overline{K_{2}}$, we deduce that $f\big{(}(u,b)\big{)}=1$, $f\big{(}(v,b)\big{)}=2$ and $f\big{(}(w,b)\big{)}=3$. We now turn our attention to the vertex $x$. Since $f$ is a $2$DMV coloring of $T\circ \overline{K_{2}}$, it follows that at least one of the vertices $f\big{(}(x,a)\big{)}$ and $f\big{(}(x,b)\big{)}$, say $f\big{(}(x,a)\big{)}$, does not equal $3$. On the other hand, $f\big{(}(x,a)\big{)}$ is neither $1$ nor $2$ because it is at distance more than $2$ from $(u,a)$ and $(v,a)$. This is a contradiction. Therefore, $\chi_{\mu_2}(T\circ \overline{K_{2}})>3=\chi_{\mu_2}(T)$.
\end{proof}


\section{Strong product of graphs}
\label{sec:stong}

We start with a specific example of strong product graphs where one factor is a path and the other is a complete graph. We will present exact values of their $k$DMV numbers for all possible values of $k$. This result will then be used to prove the sharpness of a canonical upper bound on $\hmk(G\boxtimes H)$, which is the main result of this section. 

Recall that the distance between any two vertices $(g,h),(g',h')\in V(G\strong H)$ is given by 
\begin{center}
$d_{G\strong H}\big{(}(g,h),(g',h')\big{)}=\max \{d_{G}(g,g'),d_{H}(h,h')\}$.
\end{center}

It is easy to observe that $\chi_{\mu_{1}}(P_n \strong K_m)=\theta(P_n \strong K_m)=\lceil \frac{n}{2}\rceil$ for all positive integers $n$ and $m$. Moreover, $\chi_{\mu_{k}}(P_n \strong K_m)=\chi_{\mu}(P_n \strong K_m)$ if $k\geq n-1$. In such a situation, we have $\chi_{\mu}(P_n \strong K_m)=1$ if $n\leq2$, $\chi_{\mu}(P_n \strong K_m)=2$ if $n=3$ and $m\geq2$, and $\chi_{\mu}(P_n \strong K_1)=\lceil n/2\rceil$. With this in mind, we restrict our attention to the case when $2\leq k\leq n-2$.

\begin{prop}\label{str1}
For any integers $n\geq 4$, $m\ge 2$ and $k\in \{2,\ldots,n-2\}$, 
\begin{equation*}
\chi_{\mu_k}(P_n \strong K_m)=\left \{
\begin{array}{lll}
2\big\lfloor\frac{n}{k+2}\big\rfloor & \mbox{if}\ n \equiv0\ \pmod{ k+2},\vspace{0.5mm}\\
2\big\lfloor\frac{n}{k+2}\big\rfloor+1 & \mbox{if}\ n \equiv1,2\ \pmod{k+2},\vspace{0.5mm}\\
2\big\lfloor\frac{n}{k+2}\big\rfloor+2 & \mbox{otherwise}.
\end{array}
\right.
\end{equation*}
\end{prop}
\begin{proof}
Let $P_{n}:u_{1}u_{2}\ldots u_{n}$, $V(K_{m})=\{v_1,\ldots,v_m\}$ and let $\eta=\lfloor\frac{n}{k+2}\rfloor$. Consider the vertex partition of $P_n \strong K_m$ into sets $V_i=\{(u_i,v_j)\mid j\in[m]\}$ with $i\in[n]$. Let $G_i$, with $i\in[\eta]$, be the subgraph of $P_n \strong K_m$ induced by $V_{(k+2)(i-1)+1}\cup V_{(k+2)(i-1)+2}\cup \cdots\cup V_{(k+2)i}$. Each $G_i$ can be colored by assigning $c_{i1}$ to the vertices in $V_{(k+2)(i-1)+1}$ and $m-1$ vertices from each $V_{(k+2)(i-1)+j}$ with $j\in[k+1]\setminus \{1\}$, and $c_{i2}$ to the other vertices (see Fig.~\ref{Colo}).

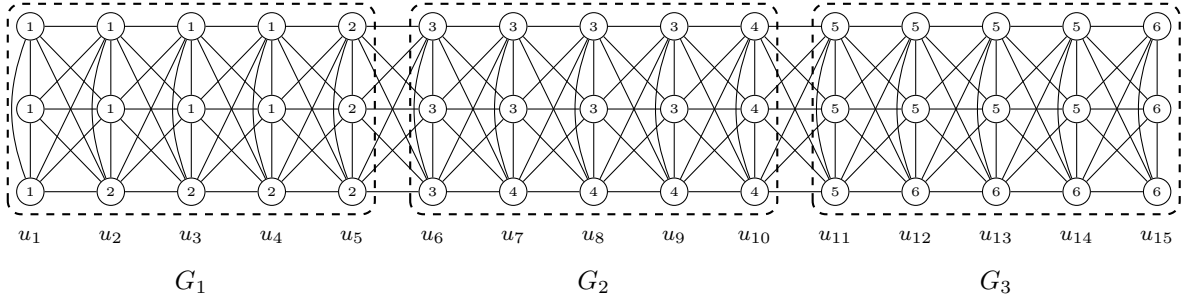
\begin{figure}[ht!]
		\centering
		\begin{tikzpicture}

			\def\stepx{1.07}
			\def\stepy{1.1}

			\foreach \i in {1,...,15} {
				\foreach \j in {1,...,3} {
				
					\pgfmathtruncatemacro{\pos}{\i - 1}
					\pgfmathtruncatemacro{\block}{ \pos / 5 }
					\pgfmathtruncatemacro{\local}{ mod(\pos, 5) }
					\ifnum\local=0
					\pgfmathtruncatemacro{\col}{2*\block + 1}
					\else\ifnum\local=4
					\pgfmathtruncatemacro{\col}{2*\block + 2}
					\else
					\ifnum\j=1
					\pgfmathtruncatemacro{\col}{2*\block + 2}
					\else
					\pgfmathtruncatemacro{\col}{2*\block + 1}
					\fi
					\fi\fi

					\node[circle, draw, inner sep=2pt, minimum size=8pt, font=\tiny] (n\i\j) at ({(\i-1)*\stepx}, {(\j-1)*\stepy}) {\col};
				}
			}
			
			\foreach \i in {1,...,15} {
	
				\draw (n\i1) -- (n\i2);
				\draw (n\i2) -- (n\i3);
			
				\draw (n\i1) to[bend left=20] (n\i3);
			}

			\foreach \i in {1,...,14} {
				\pgfmathtruncatemacro{\nexti}{\i + 1}
				\foreach \j in {1,...,3} {
					\draw (n\i\j) -- (n\nexti\j);
				}
			}

			\foreach \i in {1,...,14} {
				\pgfmathtruncatemacro{\nexti}{\i + 1}
				\foreach \j in {1,...,3} {
					\foreach \k in {1,...,3} {
						\ifnum\j=\k \else
						\draw (n\i\j) -- (n\nexti\k);
						\fi
					}
				}
			}

			\draw[rounded corners=5pt, dashed, thick] (-0.3,-0.3) rectangle (4*\stepx+0.3, 2*\stepy+0.3);
			\draw[rounded corners=5pt, dashed, thick] (5*\stepx-0.3,-0.3) rectangle (9*\stepx+0.3, 2*\stepy+0.3);
			\draw[rounded corners=5pt, dashed, thick] (10*\stepx-0.3,-0.3) rectangle (14*\stepx+0.3, 2*\stepy+0.3);

			\foreach \i in {1,...,15} {
				\node at ({(\i-1)*\stepx}, -0.6) {\footnotesize \( u_{\i} \)};
			}

			\node at ($(0, -1.2)!0.5!(4*\stepx, -1.2)$) {\( G_1 \)};
			\node at ($(5*\stepx, -1.2)!0.5!(9*\stepx, -1.2)$) {\( G_2 \)};
			\node at ($(10*\stepx, -1.2)!0.5!(14*\stepx, -1.2)$) {\( G_3 \)};
			
		\end{tikzpicture}
		\caption{\small A $3$DMV coloring of $P_{15} \strong K_3$ using $6$ colors. Vertices are partitioned into sets $V_i$, as per the vertices $u_{i}$, each forming a triangle. The subgraphs $G_1$, $G_2$ and $G_3$ are outlined with dashed rounded rectangles.}\label{Colo}
	\end{figure}

If $n\equiv0$ (mod $k+2$), then this is a $k$DMV coloring of $P_n \strong K_m$ with $2\eta$ colors. If $n\equiv1,2$ (mod $k+2$), then since the subgraph induced by $V_{n-1}\cup V_n$ is complete, we use only one additional color for the vertices in $V_{n-1}\cup V_n$. In the remaining cases, the subgraph induced by $V(P_n \strong K_m)\setminus \bigcup_{i=1}^{\eta}V(G_i)$ is not complete, but can be colored using two colors similarly to the way we colored the subgraphs $G_i$ with $i\in[\eta]$. Hence, the values given in the theorem are upper bounds for $\chi_{\mu_k}(P_{n}\strong K_{m})$ in the respective cases.

Now, let $f$ be a $\chi_{\mu_k}(P_{n}\strong K_{m})$-coloring. Note that every vertex in $V_i$ is at distance $|i-j|$ from every vertex in $V_j$. Hence, any color assigned to a vertex in $V_i$ cannot be given to any of the vertices in $V_j$ when $|i-j|>k$. If $n<k+2$, then it is readily seen that $\chi_{\mu_k}(P_{n}\strong K_{m})=1$ for $n\in \{1,2\}$, and $\chi_{\mu_k}(P_{n}\strong K_{m})=2$ otherwise. So, we may assume that $n\geq k+2$. 

We first observe that $|f\big{(}V(G_{i})|\geq2$ for each $i\in[\eta]$ and that $f\big{(}V(G_{i})\big{)}\cap f\big{(}V(G_{j})\big{)}=\emptyset$ for all $i,j\in[\eta]$ with $|i-j|>1$. Now, consider $G_{i}$ and $G_{i+1}$ for any $i\in[\eta-1]$. Since $d(u,v)>k$ and $d(v,w)>k$ for each $u\in V_{(k+2)(i-1)+1}$, $v\in V_{(k+2)i+1}$ and $w\in V_{(k+2)(i+1)}$, it follows that $|f\big{(}V(G_{i})\cup V(G_{i+1})\big{)}|\geq3$. Suppose that $|f\big{(}V(G_{i})\cup V(G_{i+1})\big{)}|=3$ for some $i\in[\eta-1]$. By definition, we may assume without loss of generality that $f(V_{(k+2)(i-1)+1})=\{1\}$, $f(V_{(k+2)i+1})=\{2\}$ and $f(V_{(k+2)(i+1)})=\{3\}$. Moreover, we necessarily have $f(V_{(k+2)(i-1)+2})=\{1\}$. In view of this, since $f$ is a $2$DMV coloring, any other vertex in $V(G_{i})\cup V(G_{i+1})$ receives a color different from $1$ under $f$. In particular, $f(x)=2$ for each $x\in V_{(k+2)(i-1)+k+1}\cup V_{(k+2)i}\cup V_{(k+2)i+1}$. This is impossible as $f$ is a $2$DMV coloring. In fact, in either case, we have shown that $|f\big{(}V(G_{i})\cup V(G_{j})\big{)}|\geq4$ for all distinct $i,j\in[\eta]$.

We deduce from the above argument that $\chi_{\mu_k}(P_n \strong K_m)\geq2\eta$. In particular, we have equality if $n\equiv0$ (mod $k+2$). We now consider two cases.\vspace{0.75mm}\\
\textit{Case 1.} $n\equiv t$ (mod $k+2$) for some $t\in \{3,\ldots,k+1\}$. We set $A=V(P_n \strong K_m)\setminus\big{(}\cup_{i=1}^{\eta-1}V(G_{i})\big{)}$, and let $A_{i}=V_{(k+2)(\eta-1)+i}$ for $i\in[k+2+t]$. Clearly, $|f(\cup_{i=1}^{3}A_{i})|\geq2$ and $|f(\cup_{i=k+t}^{k+2+t}A_{i})|\geq2$. Moreover, $f(\cup_{i=1}^{3}A_{i})\cap f(A_{k+1+t}\cup A_{k+2+t})=\emptyset$. Suppose that $|f(A)|=3$. This necessarily implies that $|f(\cup_{i=1}^{3}A_{i})|=2$ and $|f(A_{k+1+t}\cup A_{k+2+t})|=1$. Similarly, we have $|f(A_{1}\cup A_{2})|=1$ and $|f(\cup_{i=k+t}^{k+2+t}A_{i})|=2$. Due to this, we may assume that $f(A_{1}\cup A_{2})=\{1\}$ and $f(A_{k+1+t}\cup A_{k+2+t})=\{2\}$. Therefore, all vertices in $A_{3}\cup\ldots \cup A_{k+t}$ receive the same color under $f$. This contradicts the fact that $f$ is a $k$DMV coloring. Thus, $\chi_{\mu_k}(P_n \strong K_m)\geq \sum_{i=1}^{\eta-1}|V(G_{i})|+|f(A)|\geq2\eta+2$. This leads to the desired equality in this case due to the first part of the proof.\vspace{0.75mm}\\
\textit{Case 2.} $n\equiv1,2$ (mod $k+2$). Suppose to the contrary that $\chi_{\mu_k}(P_n \strong K_m)=2\eta$. This necessarily implies that $|f(A)|=2$. Without loss of generality, we may assume that $f(A)=\{1,2\}$, and that $f(A_{1})=\{1\}$ and $f(A_{k+3})=\{2\}$. Since $d(x,y)>k$ for each $x\in A_{1}$ and $y\in A_{k+2}$, it follows that $f(A_{k+2})=\{2\}$. Moreover, we have $f(A_{2})=\{1\}$ in a similar fashion. On the other hand, the rest of vertices in $A$ receive the colors $1$ or $2$ under $f$. Let $x$ be such a vertex and assume, without loss of generality, that $f(x)=2$. Then, every geodesic between $x$ and any vertex in $A_{k+3}$ has an internal vertex with color $2$, a contradiction. Therefore, $|f(A)|\geq3$. Hence, $\chi_{\mu_k}(P_n \strong K_m)\geq \sum_{i=1}^{\eta-1}|V(G_{i})|+|f(A)|\geq2\eta+1$, leading to equality in this case. This completes the proof.
\end{proof}

The following theorem establishes a canonical upper bound on $\chi_{\mu_k}(G \boxtimes H)$.

\begin{thm}\label{strong1}
For any two graphs $G$ and $H$ and any positive integer $k$, $\hmk(G \boxtimes H) \leq \hmk(G)\hmk(H)$. This bound is sharp.
\end{thm}
\begin{proof}
Let $\{V_{1},\ldots,V_{\chi_{\mu_k}(G)}\}$ and $\{W_{1},\ldots,W_{\chi_{\mu_k}(H)}\}$ be a $\chi_{\mu_k}(G)$-coloring and a $\chi_{\mu_k}(H)$-coloring, respectively. We set $\mathbb{Q}=\{V_{i}\times W_{j}\mid i\in[\chi_{\mu_k}(G)]\ \mbox{and}\ j\in[\chi_{\mu_k}(H)]\}$. Clearly, $\mathbb{Q}$ is a vertex partition of $G\boxtimes H$. We consider any set $V_{i}\times W_{j}$ in $\mathbb{Q}$ and let $(g,h)$ and $(g',h')$ be two nonadjacent vertices in $V_{i}\times W_{j}$. Without loss of generality, we may assume that $s=d_{H}(h,h')\leq d_{G}(g,g')=t$. This in particular implies that $g\neq g'$, $gg'\notin E(G)$ and that $d_{G\boxtimes H}\big{(}(g,h),(g',h')\big{)}=t$. 

Since $V_{i}$ is a $k$DMV set in $G$ and $g,g'\in V_{i}$, there exists a $g,g'$-geodesic $gg_{1}\ldots g_{t-1}g'$ in $G$ whose internal vertices do not belong to $V_{i}$ and that $t\leq k$. If $d_{H}(h,h')\leq1$, then $(g,h)(g_{1},h)\ldots(g_{t-1},h)(g',h')$ is a $(g,h),(g',h')$-geodesic in $G\boxtimes H$ of length at most $k$ with internal vertices not in $V_{i}\times W_{j}$. So, we assume that $d_{H}(h,h')\geq2$.
In a similar fashion, there is an $h,h'$-geodesic $hh_{1}\ldots h_{s-1}h'$ whose internal vertices do not intersect $W_{j}$. Hence, the path $P:(g,h)(g_{1},x_{1})\ldots(g_{t-1},x_{t-1})(g',h')$ is a $(g,h),(g',h')$-geodesic in $G\boxtimes H$ of length at most $k$, in which 
\begin{equation*}
x_{i}=\left \{
\begin{array}{lll}
h_{i} & \mbox{if}\ i\in[s-1],\vspace{0.5mm}\\
h' & \mbox{otherwise}.
\end{array}
\right.
\end{equation*}
Note that no internal vertices of $P$ belong to $V_{i}\times W_{j}$. 

The above argument shows that $V_{i}\times W_{j}$ is a $k$DMV set in $G\boxtimes H$ for all $i\in[\chi_{\mu_k}(G)]$ and $j\in[\chi_{\mu_k}(H)]$. Thus, $\chi_{\mu_{k}}(G\boxtimes H)\leq|\mathbb{Q}|=\chi_{\mu_k}(G) \chi_{\mu_k}(H)$.

The bound is sharp for infinite families of strong product graphs with respectively large and small diameters. Together Proposition~\ref{str0}$(i)$ and Proposition~\ref{str1}, with $k=2$, imply that $\lceil \frac{n}{2}\rceil=\chi_{\mu_{2}}(P_{n}\boxtimes K_{m})\leq \chi_{\mu_{2}}(P_{n})\chi_{\mu_{2}}(K_{m})=\chi_{\mu_{2}}(P_{n})=\lceil \frac{n}{2}\rceil$ for each $n\geq4$ and $m\geq1$. Moreover, for any graph $G$ with a universal vertex and integers $k\geq2$ and $m\geq1$, it is readily observed that
\begin{equation*}
\chi_{\mu_{k}}(G\boxtimes K_{m})=\left \{
\begin{array}{lll}
1 & \mbox{if}\ G\cong K_{|V(G)|},\vspace{0.5mm}\\
2 & \mbox{otherwise},
\end{array}
\right.
\end{equation*}
showing the sharpness of the bound in this case either.
\end{proof}

\begin{remark}
For any integer $r\geq1$, there exists a graph $G$ for which the strict inequality chain $\chi_{\mu_r}(G)<\chi_{\mu_{r-1}}(G)<\ldots<\chi_{\mu_2}(G)$ holds. Such a graph can be constructed, in view of Proposition~\ref{str1}, by taking $G=P_{n}\strong K_{m}$, in which $m\ge 2$ and $n=\text{lcm}(3,4,\ldots,r+2)$. 
\end{remark}


\section{Cartesian products}
\label{sec:cart}

In this section, we bound the $2$DMV chromatic number of Cartesian product graphs from below and consider this parameter specifically in the case of hypercubes.

One may wonder if there exists an upper bound on $\hmd(G\square H)$ expressed as a function of $\hmd(G)$ and $\hmd(H)$. Recall that such a result holds for the strong product of graphs, where Theorem~\ref{strong1} gives a canonical upper bound on $\chi_{\mu_k}(G\boxtimes H)$ for any positive integer $k$. In the case of Cartesian products of graphs, this definitely fails when $k=1$, because $\chi_{\mu_1}(K_n)=1$, whereas $\chi_{\mu_1}(K_n\square K_n)=n$. When $k=2$, first recall that in graphs $G$ with diameter $2$, $\hmd(G)=\chi_\mu(G)$. Now, we invoke~\cite[Theorem 3.2]{KKVY}, which states that if $n \ge 2$ is large enough, then $\chi_\mu(K_n\square K_n)=\Theta(\sqrt{n})$. Hence, $\hmd(K_n\square K_n)$ can be arbitrarily large, while $\hmd(K_n)=1$. 
\begin{ob}
There exists no function $f:\mathbb{N}\to\mathbb{N}$ such that $\hmd(G\square G)\le f(\hmd(G))$ holds for all graphs $G$.
\end{ob}

On the other hand, we present a sharp lower bound on the $2$DMV chromatic number of the Cartesian product of two graphs, for which we need another definition often used in domination theory.  
A set $S\subseteq V(G)$ is a {\em $2$-packing} in a graph $G$ if for every two vertices $u$ and $v$ in $S$, we have $N_G[u]\cap N_G[v]=\emptyset$. The maximum cardinality of a $2$-packing in $G$ is the {\em $2$-packing number}, $\rho_2(G)$, of $G$. Taking into account the definition of $2$DMV sets, we immediately infer that for every graph $G$, 
\begin{equation}
\label{eq:2pack}
\chi_{\mu_2}(G)\ge \rho_2(G).    
\end{equation}

We can use the $2$-packing number to obtain the mentioned lower bound in Cartesian product of graphs as the following result shows. 
\begin{thm}
\label{prp:cartesian}
If $G$ and $H$ are connected graphs, then 
$$\chi_{\mu_2}(G\square H)\ge \max\{\chi_{\mu_2}(G)\rho_2(H),\chi_{\mu_2}(H)\rho_2(G)\},$$
and the bound is sharp.
\end{thm}
\begin{proof}
Due to commutativity of the Cartesian product and the symmetry of the expression in the statement, it suffices to prove that $\chi_{\mu_2}(G\square H)\ge \chi_{\mu_2}(G)\rho_2(H)$. For this purpose, let $f:V(G)\times V(H)\to [\chi_{\mu_2}(G\square H)]$ and $S\subseteq V(H)$ be a $\chi_{\mu_2}(G\square H)$-function and a $\rho_2(H)$-set, respectively. Note that for each $h\in S$ the $G$-fiber $G^h$ is convex in $G\square H$, hence by Proposition~\ref{prop:convex}, we infer that $|f(G^h)|\ge \chi_{\mu_2}(G)$. Since for any distinct vertices $h$ and $h'$ in $S$, the distance between any vertex in $G^h$ and any vertex in $G^{h'}$ is at least $3$, we infer that $f(G^h)\cap f(G^{h'})=\emptyset$. Thus, we get
$$\chi_{\mu_2}(G\square H)=|f\big{(}V(G)\times V(H)\big{)}|\geq|\bigcup_{h\in S}{f(G^h)}|=\sum_{h\in S}{|f(G^h)|}\ge \sum_{h\in S}{\chi_{\mu_2}(G)}=\rho_2(H)\chi_{\mu_2}(G).$$

To see that the bound is sharp, we first consider the graph $P_4\square C_4$. Note that $\chi_{\mu_2}(C_4)=2=\chi_{\mu_2}(P_4)$, whereas $\rho_2(C_4)=1<2=\rho_2(P_4)$. Hence, $\max\{\chi_{\mu_2}(P_4)\rho_2(C_4),\chi_{\mu_2}(C_4)\rho_2(P_4)\}=4$. Figure~\ref{fig:P4C4} presents a 2DMV coloring of $P_4\square C_4$ using $4$ colors. Recall that the corona of a graph $G$, ${\rm cor}(G)$, is obtained from $G$ by attaching a leaf to each of its vertices~\cite{FH}. Note that $P_4\cong {\rm cor}(P_{2})$. Now, let $G$ be a connected graph that contains a perfect matching $M$. Let $gg'\in M$, and let $h$ and $h'$ be the leaves attached to $g$ and $g'$, respectively, in ${\rm cor}(G)$. Note that the subgraph of ${\rm cor}(G)\square C_4$ induced by $^{g}C_{4}$ $\cup$ $^{g'}\!C_{4}$ $\cup$ $^{h}C_{4}$ $\cup$ $^{h'}\!C_{4}$ is isomorphic to the graph $P_4\square C_4$, as depicted in Figure~\ref{fig:P4C4}. With this in mind, ${\rm cor}(G)\square C_4$ can be partitioned into $|M|$ disjoint copies of $P_{4}\square C_{4}$, each of which can be colored with respect to $2$DMV coloring by four colors similarly to that of Figure~\ref{fig:P4C4}. This leads to a $2$DMV coloring of ${\rm cor}(G)\square C_4$ with $4|M|=2|V(G)|$ colors. Therefore, $\chi_{\mu_2}({\rm cor}(G)\square C_4)\leq2|V(G)|$, and we end up with equality due to the proved inequality and taking $\rho_{2}\big{(}{\rm cor}(G)\big{)}=|V(G)|$ into account.   
\end{proof}

\begin{figure}[!ht]

\begin{center}
\begin{tikzpicture}[scale=.7]\footnotesize

\begin{scope}<+->;

\node[circle, draw] (x1) at (0,0) {4};
\node[circle, draw] (x2) at  (2,0) {4};
\node[circle, draw] (x3) at (4,0) {3};
\node[circle, draw] (x4) at (6,0) {3};
\node[circle, draw] (y1) at (0,2) {4};
\node[circle, draw] (y2) at (2,2) {1};
\node[circle, draw] (y3) at (4,2) {2};
\node[circle, draw] (y4) at (6,2) {3};
\node[circle, draw] (z1) at (0,4) {1};
\node[circle, draw] (z2) at (2,4) {2};
\node[circle, draw] (z3) at (4,4) {1};
\node[circle, draw] (z4) at (6,4) {2};
\node[circle, draw] (w1) at (0,6) {4};
\node[circle, draw] (w2) at (2,6) {1};
\node[circle, draw] (w3) at (4,6) {2};
\node[circle, draw] (w4) at (6,6) {3};

\draw (x1) -- (x2) -- (x3) -- (x4) -- (y4)--(y3)--(y2)--(y1)--(z1)--(z2)
--(z3)--(z4)--(w4)--(w3)--(w2)--(w1);
\draw (x1)--(y1)--(z1)--(w1);
\draw (x2)--(y2)--(z2)--(w2);
\draw (x3)--(y3)--(z3)--(w3);
\draw (x4)--(y4)--(z4)--(w4);
\draw (x1) .. controls (-0.8,0.5) and (-0.8,5) .. (w1);
\draw (x2) .. controls (1.2,0.5) and (1.2,5) .. (w2);
\draw (x3) .. controls (3.2,0.5) and (3.2,5) .. (w3);
\draw (x4) .. controls (5.2,0.5) and (5.2,5) .. (w4);

\end{scope}

\end{tikzpicture}
\end{center}
\caption{A $\chi_{\mu_2}(P_4\square C_4)$-coloring.}
\label{fig:P4C4}
\end{figure}
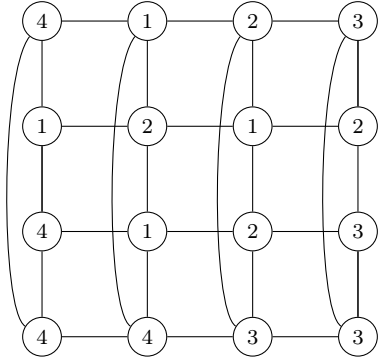

Next, we focus on the class of hypercubes. Recall that $K_2\square \cdots\square  K_2$, where there are $n$ factors in the product, is the {\em $n$-dimension hypercube}, or shortly {\em $n$-cube}, denoted by $Q_n$.   In Corollary~\ref{Cor}, we established that $\chi_{\mu_2}(G)\geq \gamma(G)$ if $G$ is a graph of girth at least $7$. Interestingly, in  hypercubes, the reversed inequality can be proved easily.

We claim that a closed neighborhood of a vertex in $Q_n$ is a $2$DMV set. Let $u\in V(Q_n)$ and $v,w\in N_{Q_{n}}(u)$. Then, $u$, $v$ and $w$ lie in a unique square in $Q_n$. Let $z$ be the fourth vertex of that square. Then, clearly, $vzw$ is a geodesic in $Q_n$ whose internal vertex $z$ is not in $N_{Q_{n}}[u]$. Hence, $N_{Q_{n}}[u]$ is a $2$DMV set, where $u$ was an arbitrarily chosen vertex in $Q_n$. The following observation immediately follows. 

\begin{prop}
\label{prp:Qn}
For any $n\in \mathbb{N}$, $\hmd(Q_n)\le \gamma(Q_n)$. 
\end{prop}
We strongly suspect that the inequality in the above proposition holds as equality. 
We base this on the following facts.

\begin{prop}
\label{ob:Qn}
 If $S\subseteq V(Q_n)$ consists of vertices that are pairwise at distance at most $2$, then either
 \begin{itemize}
     \item $S\subseteq N_{Q_{n}}[v]$ for some $v\in V(Q_n)$, or
     \item $S$ induces a square, or 
     \item $S$ is a subset of the vertex set of a $Q_3$-subgraph consisting of four vertices of one of its partite sets.  
 \end{itemize}
\end{prop}
\begin{proof}
For small values of $n$ the result is clear, so we assume that $n\ge 4$. Let $S\subseteq V(Q_n)$ consist of vertices that are pairwise at distance at most $2$ in $Q_n$. If $|S|\le 2$, the statement clearly holds, hence assume $|S|\ge 3$. Consider the representation of the vertices in $Q_n$ as binary $n$-tuples, where two $n$-tuples are adjacent in $Q_n$ whenever they differ in exactly one coordinate. Without loss of generality, we may assume that $x,y\in S$ such that $x=(0,\ldots,0,1)$ and $y=(0,\ldots,0,1,0)$. If the last two coordinates of an $n$-tuple $z\notin\{x,y\}$ are $0,1$ or $1,0$, then $z$ will be at distance greater than $2$ from $y$ or $x$, respectively. Therefore, $z\notin S$. This shows that the last two coordinates of all elements in $S\setminus \{x,y\}$ are either $0,0$ or $1,1$. If the last two coordinates of all vertices in $S\setminus \{x,y\}$ are $0,0$ (resp. $1,1$), then $S\subseteq N_{Q_n}[(0,\ldots,0)]$ (resp. $S\subseteq N_{Q_n}[(0,\ldots,0,1,1)]$) due to the distances of these vertices from $x$ (and $y$). So, assume that there are a vertex $z\in S$ with the last two coordinates $0,0$ and a vertex $w\in S$ with the last two coordinates $1,1$. If $z=(0,\ldots,0)$ and $w=(0,\ldots,0,1,1)$, then $xzywx$ is a square in $Q_n$ and we immediately infer that $S=\{x,y,z,w\}$. Thus, let us assume that $z$, which ends with $0,0$, has at least one coordinate $1$. This, in view of $d_{Q_n}(x,z)\le 2$, implies that there is exactly one such coordinate, say the third from the right. That is, $z=(0,\ldots,0,1,0,0)$. We in turn derive, due to $d_{Q_n}(z,w)\le 2$, that $w=(0,\ldots,0,1,1,1)$. Altogether, we infer that there are no other vertices in $S$ but the four mentioned ones. Moreover, $S$ is a subset of the vertex set of the $3$-cube induced by $S\cup \{(0,\ldots,0),(0,\ldots,0,1,1),(0,\ldots,0,1,1,0),(0,\ldots,0,1,0,1)\}$. We observe that $S$ is a partite set of this $3$-cube. 
\end{proof}

Now, if $C$ is a color class of a 2DMV coloring of $Q_n$, one of the three items of Proposition~\ref{ob:Qn} has to hold for $C$. In fact, the second item is not possible, since for the two opposite vertices of a square there are exactly two geodesics which pass through the other vertices of the square. If for all color classes of a $\hmd(Q_n)$-coloring the first item holds, then $\hmd(Q_n)=\gamma(Q_n)$. Hence, one needs to deal with potential color classes of a $\hmd(Q_n)$-coloring, which consist of only four vertices forming one of the partite sets of a $Q_3$-subgraph of $Q_n$. The mentioned arguments are the basis for our suspicion that $\hmd(Q_n)=\gamma(Q_n)$ could hold for all $n\in \mathbb{N}$.


\section{Block graphs}
\label{sec:block}

Recall that a {\em block} in a graph $G$ is a maximal connected subgraph of $G$ that has no cut-vertex. A graph is a {\em block graph} if all its blocks are cliques. The {\em eccentricity} of a vertex $v\in V(G)$ is defined as $\epsilon_{G}(v)=\max\{d_G(u,v)\mid u\in V(G)\}$. The {\em radius} of $G$, written ${\rm rad}(G)$, is $\min\{\epsilon_{G}(v)\mid v\in V(G)\}$. The {\em center}, $C(G)$, of $G$ is the subgraph induced by the vertices of minimum eccentricity. A vertex $u\in V(G)$ is a {\em radial vertex} of $G$ if there exists a vertex $x\in C(G)$ such that $d_G(u,x)={\rm rad}(G)$.

It was proved in~\cite[Proposition 5.1]{KKVY} that for every block graph $G$,
\begin{equation}
\hm(G)=\left\lceil\frac{{\rm diam}(G)+1}{2}\right\rceil.
    \label{eq:block}
\end{equation}

Recall that for any graph $G$, $\hmk(G)=\hm(G)$ as soon as $k\ge {\rm diam}(G)$. In this section, we characterize block graphs $G$ for which the equality $\hmk(G)=\hm(G)$ extends to the case $k={\rm diam}(G)-1$. 

It is known that the center of a block graph $G$ is a clique. More precisely, two possibilities occur: either the center is a vertex, which is intersected by several blocks in $G$, or it is a block itself. Note that when $|C(G)|=1$, ${\rm diam}(G)=2{\rm rad}(G)$, and so  $\left\lceil\frac{{\rm diam}(G)+1}{2}\right\rceil={\rm rad}(G)+1$. On the other hand, when $|C|\ge 2$, we have ${\rm diam}(G)=2{\rm rad}(G)-1$, so in this case $\left\lceil\frac{{\rm diam}(G)+1}{2}\right\rceil={\rm rad}(G)$. In the characterization of block graphs $G$ in which ${\chi_{\mu_{d-1}}}(G)=\hm(G)$, where $d={\rm diam}(G)$, we will need the following notion. 

Let $C$ be the center of a graph $G$. If $|C|\ge 2$ and $G[C]$ is the corresponding subgraph in $G$, then let $\deg^\ast(C)$ denote the number of components of $G-E(G[C])$ that contain a radial vertex of $G$. If $C=\{c\}$ and $F$ is the set of edges incident with $c$, then let $\deg^\ast(C)$ denote the number of components of $G-F$ that contain a radial vertex of $G$.

\begin{thm}
\label{thm:block}
Let $G$ be a block graph with center $C$, and let $d={\rm diam}(G)$. Then, $\chi_{\mu_{d-1}}(G)=\hm(G)$ if and only if $$\deg^\ast(C)\le \left\lceil\frac{d+1}{2}\right\rceil$$
\end{thm}
\begin{proof} 
We first check that the condition for achieving $\chi_{\mu_{d-1}}(G)=\hm(G)$ is necessary. Clearly, if two radial vertices $u$ and $v$ belong to two distinct components of $G-C$, then they are at distance $d$ in $G$. Hence, they need to receive different colors by any $(d-1)$DMV coloring. Therefore, only to color the radial vertices that belong to pairwise distinct components of $G-C$, one needs at least $\deg^\ast(C)$ colors. Thus, 
$$\left\lceil\frac{d+1}{2}\right\rceil=\hm(G)=\chi_{\mu_{d-1}}(G)\ge \deg^\ast(C),$$
as desired.

To prove the sufficiency, let us first consider the case when $|C|\ge 2$. As mentioned  earlier, in this case, $\hm(G)={\rm rad}(G)$. Setting $r={\rm rad}(G)$,
let us assume that $\deg^\ast(C)\le r$. Consider the vertex partition of $G$ into levels depending on the distance from the center $C$ as follows: 
$$L^{(i)}=\{x\in V(G)\mid d_G(x,C)=i\},$$
where $i\in \{0,1,\ldots,r-1\}$. Note that $L^{(0)}=C$ induces a clique and that $G-E(G[C])$ consists $|C|$ components, which we denote by $C_1,\ldots C_{|C|}$. Among these components, exactly $\deg^\ast(C)$ components contain a radial vertex, and we may choose the notation in such a way that these are $C_1,\ldots, C_p$, where $p={\deg^\ast(C)}$. For each $i\in \{0,1,\ldots,r-1\}$ and $j\in[|C|]$, let $L_j^{(i)}=C_j\cap L^{(i)}$. In particular, letting $L_j^{(0)}=\{c_j\}$ for all $j\in[|C|]$, we have $C=\{c_1,\ldots,c_{|C|}\}$. 
Note that $L_j^{(r-1)}\ne \emptyset$ if and only if $j\le p$. On the other hand, for $j>p$, some of the sets $L_j^{(i)}$ can be empty. 

Let us define a coloring $f:V(G)\to [r]$, which we will argue is a $(d-1)$DMV coloring. First, let $f(L_i^{(r-1)})=\{i\}$ for all $i\in[p]$. Next, for all $i\in [p]$ and every $j>i$, let $f(L_j^{(r-1-i)})=\{i\}$ as soon as $L_j^{(r-1-i)}\ne\emptyset$. Then, for all $i\in\{2,\ldots,p\}$ and $j\in[i-1]$, let $f(L_j^{(r-i)})=\{i\}$. 
Now, there are two possibilities. If $p=r$, all vertices of $G$ have already been colored. On the other hand, if $p<r$, then let $f(x)=r-i$ for all $i\in\{0,\ldots,r-p-1\}$ and all vertices $x\in L^{(i)}$ which have not already been colored by $f$. (For an illustration see Fig.~\ref{fig:block} depicting a block graph with the center $C=\{c_1,\ldots,c_4\}$ and the coloring just presented.)

It is straightforward to verify that $f$ is a mutual-visibility coloring. To see that is is also a $(d-1)$DMV coloring, one has to note that the only vertices that are at distance $d$ are such vertices $u$ and $v$, where $u\in L^{(r-1)}_{j}$, $v\in L^{(r-1)}_{j'}$ and $j\ne j'$. However, by the definition of $f$, we have $f(u)=j\neq j'=f(v)$.

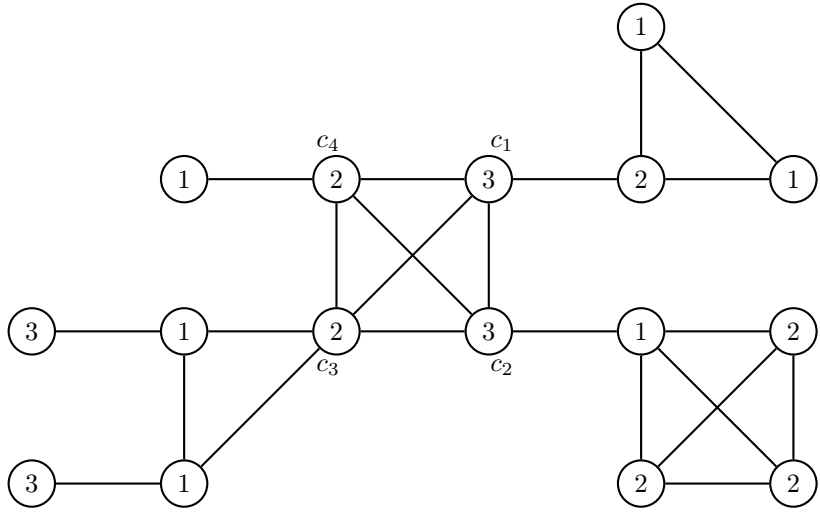
\begin{figure}[htb]
\begin{center}
\begin{tikzpicture}[scale=0.45, style=thick, x=1.5cm, y=1.5cm]

\node[circle, draw] (c3) at (0,0) {2};
\node[circle, draw] (c2) at (3,0) {3};
\node[circle, draw] (c1) at (3,3) {3};
\node[circle, draw] (c4) at (0,3) {2};
\node[circle, draw] (d2) at (6,0) {1};
\node[circle, draw] (d3) at (9,0) {2};
\node[circle, draw] (d3a) at (9,-3) {2};
\node[circle, draw] (d3b) at (6,-3) {2};
\node[circle, draw] (a2) at (6,3) {2};
\node[circle, draw] (a3a) at (9,3) {1};
\node[circle, draw] (a3b) at (6,6) {1};
\node[circle, draw] (b2a) at (-3,0) {1};
\node[circle, draw] (b2b) at (-3,-3) {1};
\node[circle, draw] (b3a) at (-6,0) {3};
\node[circle, draw] (b3b) at (-6,-3) {3};
\node[circle, draw] (u) at (-3,3) {1};

\node[left=0.7cm] at (1.3,-0.7) {$ c_3$};
\node[right=0.7cm] at (1.8,-0.7) {$ c_2$};
\node[right=0.7cm] at (1.8,3.7) {$ c_1$};
\node[left=0.7cm] at (1.3,3.7) {$ c_4$};

\draw (c1) -- (c2) -- (c3) -- (c4) -- (c2);
\draw (c4) -- (c1) -- (c3);
\draw (c2) -- (d2) -- (d3) -- (d3a) -- (d3b);
\draw (d3) -- (d3b) -- (d2) -- (d3a);
\draw (c1) -- (a2) -- (a3a) -- (a3b);
\draw (a2) -- (a3b);
\draw (c4) -- (u);
\draw (b2b) -- (b2a) -- (c3) -- (b2b) -- (b3b);
\draw (b3a) -- (b2a);

\end{tikzpicture}

\caption{A block graph $G$, and the $\hmd(G)$-coloring as given in the proof of Theorem~\ref{thm:block}. Note that the vertices of the center $C$ are labeled, and $\deg^*(C)={\rm rad}(G)=3$. }
\label{fig:block}
\end{center}
\end{figure}

The case when $|C|=1$ can be treated by using a similar partition as above. Note that in this case, $\hm(G)={\rm rad}(G)+1=r+1$, and therefore the assumption is $\deg^\ast(C)\le r+1$. Letting $C=\{c\}$, consider the following partition of $V(G)$ into levels depending on the distance from $c$: 
$$L^{(i)}=\{x\in V(G):\,d_G(x,c)=i\},$$
where $i\in \{0,1,\ldots,r\}$. Note that $L^{(0)}=\{c\}$. Let $F$ be the set of edges incident with $c$, and let $C_1,\ldots C_t$ be the components of $G-F$. By the assumption, $p=\deg^\ast(C)$ of these components contain a radial vertex of $G$, and let them be denoted by $C_1,\ldots,C_p$.
 For each $i\in [r]$ and $j\in [t]$, let $L_j^{(i)}=C_j\cap L^{(i)}$.  
Note that $L_j^{(r)}\ne \emptyset$ if and only if $j\le p$. On the other hand, for $j>p$ some of the sets $L_j^{(i)}$ can be empty. 

We define a coloring $f:V(G)\to [r]$ in a similar way as in the case $|C|\ge 2$. 
We let $f(L_i^{(r)})=\{i\}$ for all $i\in [p]$. Next, for all $i\in[p-1]$ (respectively, $i\in [p-2]$, when $p=r+1$) and every $j>i$, let $f(L_j^{(r-i)})=\{i\}$ as soon as $L_j^{(r-i)}\ne\emptyset$. Then, for all $i\in\{2,\ldots,p\}$ (respectively, $i\in\{2,\ldots,p-1\}$, when $p=r+1$) and $j\in[i-1]$ let $f(L_j^{(r-i+1)})=\{i\}$.  
If $p=r+1$, then let $f(c)=p$. Otherwise, if $p\leq r$, then let $f(L^{(i)})=\{r-i+1\}$ for all $i\in\{0,\ldots,r-p\}$. Again, one can check that $f$ is a $(d-1)$DMV coloring, which completes the proof. 
\end{proof}

\section{Concluding remarks}

In this paper, we initiate the study of the $k$-distance mutual-visibility chromatic number. As with every new concept, there are many open problems and directions that one can explore. Instead of trying to present many such possibilities for further research, we restrict our attention to open problems and aspects that arise from the results in this paper. 
 
Based on the results from the previous section, it would  be interesting to consider the $2$DMV chromatic number of block graphs. In particular, the following problem seems to be challenging.

\begin{prob}
  Characterize the block graphs $G$ with $\hmd(G)=\theta(G)$.  
\end{prob}
\noindent If the above problem is too difficult, its restriction to trees could also give an interesting characterization. 

We wonder whether the bound in Theorem~\ref{prp:cartesian} can be improved by replacing the $2$-packing number in the expression of the lower bound with the domination number of the corresponding factor. We pose this as another open problem. 

\begin{prob}
\label{pr:cartesian}
Is it true that for every two connected graphs $G$ and $H$,    
$$\chi_{\mu_2}(G\square H)\ge \max\{\chi_{\mu_2}(G)\gamma(H),\chi_{\mu_2}(H)\gamma(G)\}?$$
\end{prob}

Observation~\ref{ob:2} gave the general lower bound $\chi_{\mu_k}(G)\geq \left\lceil \frac{|V(G)|}{\mu_k(G)} \right\rceil$ holding for all graphs $G$. The sharpness of the bound can be presented in the case $k=2$ by using what we think is an interesting connection with efficient open domination graphs, the concept studied in~\cite{kpy}. A graph $G$ is an {\em efficient open domination graph} if there exists a subset $D$ (called an efficient open dominating set) of $V(G)$ for which the neighborhoods centered in vertices of $D$ form a partition of $V(G)$.  

Consider a torus graph, $G=C_m\square C_n$, where $m\ge n\ge 4$. As noted in~\cite{kpy}, a torus graph $C_m\square C_n$ is an efficient open domination graph if $m\equiv 0 \pmod 4$ and $n\equiv 0\pmod 4$. In such a case, the mentioned lower bound gives us $\chi_{\mu_2}(C_m\square C_n)\ge \frac{mn}{4}$. 
 Now, given an efficient open dominating set $D=\{x_1,\ldots,x_{mn/4}\}$ of $G$, we let $A_i=N_{G}(x_i)$ for all $i\in[mn/4]$. By the definition of an efficient open dominating set, $A_{i}\cap A_{j}=\emptyset$ for $i\ne j$, and $\bigcup_{i=1}^{mn/4}{A_i}=V(G)$. Hence, the partition $\{A_1,\ldots,A_{mn/4}\}$ is a 2DMV coloring, and so, $\chi_{\mu_2}(C_m\square C_n)\le \frac{mn}{4}$. We infer that $$\chi_{\mu_2}(C_m\square C_n)=\frac{mn}{4},$$
 holds for all $m\equiv 0 \pmod 4$ and $n\equiv 0\pmod 4$.

In Section~\ref{sec:cart}, we mentioned our suspicion regarding the $2$DMV chromatic number of hypercubes, and we end the paper with the corresponding question. 

\begin{ques}
Does $\hmd(Q_n)=\gamma(Q_n)$ hold for all $n\in \mathbb{N}$? 
\end{ques}

\section*{Acknowledgments}
S.B. thanks Kerala State Council for Science,Technology and Environment (KSCSTE) for the financial support.
B.B. acknowledges the financial support of the Slovenian Research and Innovation Agency (research core funding No.\ P1-0297, projects N1-0285, N1-0431).


\end{document}